\documentclass[11pt]{article}
\usepackage[letterpaper, margin=1.0in]{geometry}

\usepackage{times}
\usepackage{microtype}
\usepackage{fullpage}
\usepackage[numbers,sort]{natbib}

\usepackage[dvipsnames,svgnames]{xcolor} 
\usepackage{tikz}
\usetikzlibrary{shapes}
\usetikzlibrary{positioning}
\usetikzlibrary{plotmarks}

\usepackage{microtype}
\usepackage{graphicx}
\usepackage{caption}
\usepackage{booktabs} 
\usepackage{multirow}
\usepackage{adjustbox}

\usepackage{stmaryrd}
\usepackage{graphicx}
\usepackage{caption}
\usepackage{subcaption}

\usepackage{textcomp} 
\usepackage{quoting}
\usepackage{titletoc} 
\usepackage{fontawesome} 

\usepackage[colorlinks = true,
            linkcolor = blue,
            urlcolor  = blue,
            citecolor = blue,
            anchorcolor = blue]{hyperref}

\newlength{\continueindent}
\usepackage{etoolbox}
\makeatletter
\newcommand*{\ALG@customparshape}{\parshape 2 \leftmargin \linewidth \dimexpr\ALG@tlm+\continueindent\relax \dimexpr\linewidth+\leftmargin-\ALG@tlm-\continueindent\relax}
\makeatother

\providecommand{\keywords}[1]{\textbf{\textit{Keywords }} #1}

\usepackage{graphicx}
\usepackage{amsmath,amsthm,amssymb,mathtools,graphicx}
\usepackage{dsfont}
\usepackage[mathscr]{euscript}
\usepackage[linesnumbered,algoruled,boxed,lined]{algorithm2e}

\newtheorem{thm}{Theorem}
\newtheorem{example}{Example}[section]
\newtheorem{proposition}[thm]{Proposition}

\newenvironment{acknowledgements}
    {\large\bfseries Acknowledgement%
    \par\medskip\normalfont\normalsize}%
    {}%

\makeatletter
\newcommand{\nosemic}{\renewcommand{\@endalgocfline}{\relax}}
\newcommand{\dosemic}{\renewcommand{\@endalgocfline}{\algocf@endline}}
\newcommand{\pushline}{\Indp}
\let\oldnl\nl
\newcommand{\nonl}{\renewcommand{\nl}{\let\nl\oldnl}}
\makeatother
\SetKwInput{KwInit}{Initialization}
\usepackage{url}

\usepackage[dvipsnames]{xcolor} 
\usepackage{tikz}
\usetikzlibrary{shapes}
\usetikzlibrary{positioning}
\usetikzlibrary{plotmarks}
\usepackage{pgfplots}
\pgfplotsset{compat=1.7}

\usepackage{booktabs}

\usepackage{booktabs} 
\usepackage{multirow}
\usepackage{adjustbox}

\usepackage{subcaption}
\DeclareMathOperator*{\argmax}{arg\,max}

\DeclareMathOperator*{\conv}{conv}
\DeclareMathOperator*{\jaco}{J}
\DeclareMathOperator*{\e}{e}

\newcommand{\x}{w}
\newcommand{\w}{w}

\newcommand{\fduchi}{\widetilde f}
\newcommand{\Dduchi}{\mathcal{S}}

\newcommand{\Sp}{\bar Q_p}
\newcommand{\Qp}{Q_p}

\DeclareMathOperator*{\R}{\mathbb{R}}
\DeclareMathOperator*{\RR}{\mathbb{R}}
\DeclareMathOperator*{\PP}{\mathbb{P}}
\DeclareMathOperator*{\EE}{\mathbb{E}}
\DeclareMathOperator*{\Rd}{\mathbb{R}^d}

\DeclareMathOperator*{\Rp}{\mathbb{R}^p}
\DeclareMathOperator*{\Rq}{\mathbb{R}^q}

\newcommand{\mathbbm}[1]{\mathds{1}}
\newcommand{\jac}{\jaco\!L}

\usepackage{listings}
\usepackage{color}

\definecolor{dkgreen}{rgb}{0,0.6,0}
\definecolor{gray}{rgb}{0.5,0.5,0.5}
\definecolor{mauve}{rgb}{0.58,0,0.82}

\usepackage[linesnumbered,algoruled,boxed,lined]{algorithm2e}
\SetKwInput{KwInit}{Initialize}
\usepackage{hyperref}

\usepackage{color}
\definecolor{deepblue}{rgb}{0,0,0.5}
\definecolor{deepred}{rgb}{0.6,0,0}
\definecolor{deepgreen}{rgb}{0,0.5,0}

\usepackage{listings}
\usepackage{verbatim}

\title{Superquantile-based learning:\\a direct approach using gradient-based optimization\thanks{A preliminary version of this work~\cite{laguel2020first} was presented at the IEEE MLSP conference in September 2020. This work is based on Y. Laguel's MSc. thesis defended in Summer 2018.}}

\author{
Yassine Laguel$^{1}$, Jérôme Malick$^{1}$, Zaid Harchaoui$^{2}$ \\
\small{$^{1}$Univ. Grenoble Alpes, CNRS, Grenoble INP, LJK, 38000 Grenoble, France}  \\
\small{$^{2}$University of Washington, Seattle, WA, USA}}

\begin{document}
\maketitle



\begin{abstract}
We consider a formulation of supervised learning that endows models 
with robustness to distributional shifts from training to testing. The formulation hinges upon the 
superquantile risk measure, also known as the conditional value-at-risk, 
which has shown promise in recent applications of machine learning and signal processing.
We show that, thanks to a direct smoothing of the superquantile function, a superquantile-based learning objective is amenable to gradient-based optimization, using batch optimization algorithms such as gradient descent or quasi-Newton algorithms, or using stochastic optimization algorithms such as stochastic gradient algorithms. A companion software~\texttt{SPQR} implements in Python the algorithms described and allows practitioners to experiment with superquantile-based supervised learning. 
\end{abstract}


\section{Introduction: Superquantile comes into play}
\label{sec:intro}
Classical supervised learning via empirical risk (or negative log-likelihood) minimization relies on the assumption that the testing distribution coincides with the training distribution. This assumption can be challenged in domain applications of machine learning such as visual systems or dialog systems~\cite{recht2019imagenet}. Learning machines may then operate at prediction time with testing data whose distribution departs from the one of the training data. 
Recent failures of learning systems when operating in unknown environments~\cite{metz2018microsoft,knight2018selfdriving} underscore the importance of reconsidering the learning objective used to train learning machines in order to ensure robust behavior in the face of  prevalence of worst-case scenarios or unexpected distributions at prediction time. 

The generalized regression framework presented in \cite{rockafellar2008risk} provides an attractive ground to design learning machines displaying increased robustness. This framework hinges upon modeling worst-case aversion with superquantile, also known as Conditional Value-at-Risk, a statistical summary of 
the tail of the distribution considered~\cite{lee2018minimax,duchi2019variance,kuhn2019wasserstein}. The superquantile stands out as one of prominent examples of 
risk measures, well-studied in economics and finance~\cite{rockafellar2000optimization,ben2007old}. The superquantile has recently drawn an increasing attention in machine learning; see e.g.~fair learning~\cite{willaimson2019fairness}, federated learning~\cite{laguel:device}, adversarial classification~\cite{ho2020adversarial}, submodular optimization \cite{wilder2018risk}, and reinforcement learning~\cite{chow2015risk} among others.

The notion of robustness brought by the superquantile is aligned with the one in distributionally robust optimization~\cite{ben2009robust} and empirical likelihood estimation~\cite{owen2001empirical}. It is, however, different, from notions of robustness commonly considered in robust statistics~\cite[Sec. 12.6]{ben2009robust}.
The superquantile provides an efficient and mathematical-grounded adaptive re-weighting scheme of the training data, allowing one to learn predictive models with better worst-case performances that standard models obtained from empirical risk minimization. This has been corroborated empirically by a number of recent papers; see e.g.~\cite{willaimson2019fairness,laguel:device,levy2020large,curi2020adaptive,soma2020statistical}.
Recent work~\cite{duchi2018learning} established learning-theoretic generalization bounds for statistical models trained through the minimization of related objectives.

Despite attractive theoretical and practical properties, 
superquantile-based learning may be less developed than it could have been in machine learning and signal processing. This may be due to the lack of (i) direct scalable algorithms for superquantile-based optimization and (ii) easy-to-use software packages to benchmark superquantile optimization algorithms.

\vspace*{-2ex}

{
\paragraph{Contributions of this work.}
In this paper, we present a publicly-available and easy-to-use Python toolbox for superquantile-based learning, building off the popular software library\;\texttt{scikit-learn}. This paper is a follow-up of our IEEE MLSP 2020~conference paper~\cite{laguel2020first}, incorporating recent work in an extended literature review, providing additional features to the toolbox, and presenting further empirical illustrations of the robustness brought by superquantile.

More precisely, the contributions of this work are the following:
\begin{itemize}
    \item We provide a gentle introduction to superquantile-based learning. We present the main notions; we 
    review several choices of optimization algorithms; we also discuss the various numerical components used explicitly or implicitly in recent papers. These components include for instance various strategies to overcome the non-smoothness inherent to the superquantile function. 
    \smallskip
    
    \item We provide elementary analyses as well as template routines within a companion software package. 
    We primarily focus on operational aspects and give pointers to recent theoretical developments.
    
    \smallskip
    
    \item We provide numerical experiments illustrating (i) the interest of using batch quasi-Newton optimization algorithms for minimizing superquantile-based objectives and (ii) the robustness of superquantile-based models compared to the standard models obtained from empirical risk minimization.
\end{itemize}}

\vspace*{-1ex}
\paragraph{Outline of the paper.}
The outline of the paper is as follows. 
We set the stage by formalizing, in Section~\ref{sec:setting}, the framework of superquantile-based supervised learning, highlighting the three classical formulations of superquantile-based objectives. In Section~\ref{sec:oracles}, we study the differentiability of these objective functions, provide practical expressions of their (sub)gradients, together with fast procedures to compute them. In Section~\ref{sec:algos}, we overview batch and mini-batch first-order methods using these fast oracles. In Section~\ref{sec:spqr}, we provide a short presentation of the toolbox~\texttt{SPQR} 
for superquantile-based learning. {Finally, we illustrate in Section~\ref{sec:numexp} the interests of superquantile and \texttt{SPQR} for robustness in standard regression/classification tasks.}

\vspace*{-2ex}

{\paragraph{Most important related work.}
The 
introduction has already mentioned a variety of works related to superquantile, robustness, and applications in machine learning and signal processing. Finally, we highlight here the most important 
articles, in view of the contributions of this work, regarding the algorithms for superquantile optimization and the interest of superquantile in learning. 

Classical approaches for superquantile-based optimization consider convex programming techniques, including interior point algorithms; see the review of\;\cite{rockafellar2014superquantile}. The use of first-order algorithms in this context is quite recent and seems to be driven by machine learning considerations.
A key reference for our work is \cite{levy2020large}
which introduces an efficient approximated stochastic gradient algorithm for superquantile-based learning. We have implemented this algorithm within our toolbox and 
compared it with a simple approach using batch quasi-Newton method (in Section~\ref{sec:exp1}).

The interest of using superquantile in learning has been shown empirically in several recent papers, including \cite{willaimson2019fairness,laguel:device,levy2020large,curi2020adaptive,soma2020statistical}. In particular \cite{willaimson2019fairness}, studying fairness issues, empirically demonstrates that superquantile trades predictive accuracy for less fairness violation. 
In a context of federated learning, \cite{laguel:device} compares the performances of models learned 
by superquantile-based learning to standard models: for heterogeneous data, significant improvements on worst cases are reported for both error testing and accuracy on classification tasks. In our numerical experiments, we use similar representations to visualize the impact of the superquantile. Our experimental results align with those of
\cite{curi2020adaptive}, where the robustness of superquantile models on distributionally shifted datasets is demonstrated.}

\newpage
\section{Superquantile-based learning framework}
\label{sec:setting}
We are interested in a supervised machine learning setting with training data $\mathcal D = (x_i,y_i)_{1 \leq i \leq n} \in (\Rp \times \Rq)^n$, 
a prediction function $\varphi: \Rd \times \Rp \rightarrow \Rq$ (such as an additive model or a neural network) and a loss function $\ell: \Rq \times \Rq \rightarrow \R$ (such as the logistic loss or the least-squares loss).  
{Denoting\;$w\in \Rd$ the parameter (``weights") to be optimized}, the classical empirical risk minimization (ERM) problem reads
\begin{equation}\label{eq:ERM}
\min_{w \in \Rd} ~~{\frac{1}{n}\sum^n_{i=1}\ell(y_i, \varphi(w,x_i)}=\mathbb{E}_{(x,y)\sim \mathcal{D}}\left(\ell(y, \varphi(w,x))\right) ,
\end{equation}
In the above expression as expectation, we identify, by abuse of notation, the training data $\mathcal D$ with the empirical measure of the training data. With this ERM problem, we aim at achieving a small loss with an equal weighting across all training data-points. In the event that, at testing time, some probability mass gets shifted from a fraction of them onto another, large losses may then be incurred.

In order to be robust against such uncertainty in the way probability mass will spread at testing time, we can consider, instead, a training objective that involves a minimization problems with respect to a pessimistic re-weighting of the training datapoints. This boils down to replacing the expectation in \eqref{eq:ERM} by a tail-sensitive or 
\emph{risk-sensitive} quantity.
Risk-sensitive measures play a crucial role in optimization under uncertainty. Among popular convex risk measures, the superquantile, also called Conditional Value at Risk, has received particular attention because of its nice convexity properties; 
we refer to the seminal work\;\cite{rockafellar2000optimization} and the classical 
textbook~\cite[Chap.\;6]{shapiro2014lectures}. 

We use here the notation and terminology of\;\cite{rockafellar2013superquantiles}.
Consider a probability space\;$\Omega$, with probability denoted\;$\PP$. 
For any $p \in [0,1]$, the $p$-quantile of a random variable $U\colon\Omega\rightarrow\RR$, denoted by $Q_p(U)$,
is the inverse of the cumulative distribution function of $U$: for all $t\in \RR$ we have 
\begin{equation}\label{eq:Qp}
\Qp(U) \leq t \iff \PP(U \leq t) \geq p\,.
\end{equation}
The $p$-superquantile of $U$ is then defined as the mean of values of quantiles greater than a threshold $p$
\begin{equation}\label{eq:def_cvar}
\Sp(U) = \frac{1}{1-p} \int_{s=p}^1 Q_{s}(U) \mathrm{d}s\,.
\end{equation}
The analogue to \eqref{eq:Qp} for the superquantile is stronger:
\[
\Sp(U) \leq  t \iff \text{$U$ is lower than $t$ on average in its $p$-tail.}
\]
The superquantile can be therefore interpreted as a measure of the upper tail of the distribution of $U$ with the parameter $p$ controlling the sensitivity to high losses.

\begin{figure}[t!]
  \centering
  \vspace{2mm}
  \resizebox{0.5\linewidth}{!}{



\begin{tikzpicture}[
    declare function={gamma(\z)=
    2.506628274631*sqrt(1/\z)+ 0.20888568*(1/\z)^(1.5)+ 0.00870357*(1/\z)^(2.5)- (174.2106599*(1/\z)^(3.5))/25920- (715.6423511*(1/\z)^(4.5))/1244160)*exp((-ln(1/\z)-1)*\z;},
    declare function={gammapdf(\x,\k,\theta) = 1/(\theta^\k)*1/(gamma(\k))*\x^(\k-1)*exp(-\x/\theta);}
]
\pgfmathsetmacro{\mean}{4} 
\pgfmathsetmacro{\q}{5.98} 
\pgfmathsetmacro{\sq}{8.48} 
\begin{axis}[
  no markers, domain=0:9, samples=100,
  axis lines=left, xlabel=\empty, ylabel=\empty,
  every axis y label/.style={at=(current axis.above origin),anchor=east},
  every axis x label/.style={at=(current axis.right of origin),anchor=north},
  height=5cm, width=9cm,
  extra x ticks={\mean, \q},
  extra x tick labels={\empty, \empty},
  xtick={0}, ytick=\empty,
  xticklabels=\empty,
  enlargelimits=false, clip=false, axis on top
  ]

\addplot [very thick,cyan!20!black,domain=0:20] {gammapdf(x,2,2)};
\addplot [fill=white!20!white, draw=none, domain=0.:\q] {gammapdf(x,2,2)} \closedcycle;
\addplot [very thick, preaction={fill=red!70, fill opacity=0.2}, fill opacity=0.2, fill=cyan!20, draw=none, domain=\q:20] {gammapdf(x,2,2)} \closedcycle;

\draw[style=dashed] (axis cs:\mean,0.0) -- (axis cs:\mean,{gammapdf(\mean,2,2)});
\draw (axis cs:\q,0.0) -- (axis cs:\q,{gammapdf(\q,2,2)});
\draw[style=dashdotted,->, thick] (axis cs:\sq, 0.0) -- (axis cs:\sq, {gammapdf(\sq,2,2)});

\node[anchor=west] at (axis cs:{\mean}, 0.14) {$\mathbb{E}[U]$};
\node[anchor=west] at (axis cs:{\q}, 0.08) {$Q_p(U)$};
\node[anchor=west] at (axis cs:{\sq-0.5}, 0.045) {$\bar Q_p(U) = \frac{1}{1-p} \int_{s=p}^1 Q_{s}(U) \mathrm{d}s$};
\end{axis}
\end{tikzpicture}}
\vspace{-2mm}
\caption{Illustration of the expectation $\mathbb{E}(U)$, the $p$-quantile $Q_p(U)$, and the $(1\!-\!p)$-super\-quantile $\bar Q_p(U)$ of a random variable $U$.}
\label{images:superquantile}
\end{figure}
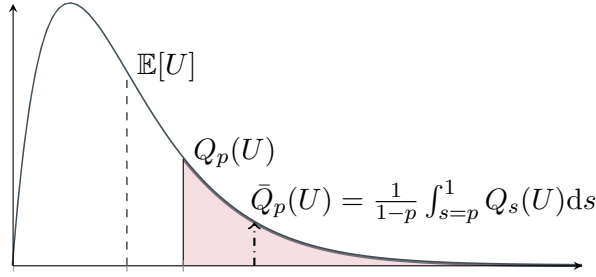 

In the case where the random variable $U$ takes equi-probable realizations $u_1,\ldots,u_n$, 
the integral \eqref{eq:def_cvar} reduces to an average of the $u_i$ that are greater or equal than the quantile. This sum can be further split in two parts with the $u_i$ that are equal to the quantile and those are that strictly larger (indexed by\;$I_>$). Mathematically, this writes
\begin{equation}\label{eq:integral_discrete}
    \Sp(U) = \frac{1}{n(1\!-\!p)}\!\! \sum_{~i \in I_>} \!\!u_i  
        + \frac{\delta}{1\!-\!p} \Qp(U)
        ~~~~~\text{with $I_> \!= \{i: u_i\!>\!\Qp(U)\}$.}
\end{equation}
where  
$
    \delta = F_U(Q_p(U))-p = \frac{1}{n} (n - |I_>|) - p.\,
$
 This expression involves the distance from $p$ to the next discontinuity point of the quantile function. Thus,
\eqref{eq:integral_discrete} provides a direct way to compute the superquantile from the computation of the quantile.

Going back to the context of learning described at the beginning of this section, we consider the superquantile of discrete distributions standing for the training data, that we denote by $[{\bar Q}_{p}]_{(x,y)\sim \mathcal{D}}$. 
A risk-sensitive statistical learning framework using the superquantile of losses rather than the expected loss thus formally replaces in \eqref{eq:ERM} the expectation by the superquantile
\begin{equation}\label{eq:ESM}
\min_{w \in \Rd} ~~f(w) = {[{\bar Q}_{p}]}_{(x,y)\sim \mathcal{D}}\big(\ell(y, \varphi(w,x))\big).
\end{equation}

This superquantile-based objective function has some special properties. First is has a nice variational formulation\;\cite{rockafellar2000optimization}:
\begin{equation}\label{eq:fminmin}
f(\w)=   
\min_{\eta \in \RR} ~\left\{\eta + \frac{1}{n(1-p)} \sum_{i=1}^n \max\{\ell(y_i, \varphi(w,x_i)) - \eta , 0\} \right\}.
\end{equation}
{This formulation opens the way to treating \eqref{eq:ESM} as a joint minimization over $(w,\eta)$; this is discussed in Section\;\ref{sec:algos}. Note here that} the minimization with respect to $\eta$ in \eqref{eq:fminmin} exactly gives the $p$-quantile of the losses and can be done efficiently in linear time. 

Using standard duality, we can also write 
the 
min problem \eqref{eq:fminmin} as a max, which takes the form
\begin{equation}\label{eq:fminmax}
f(\w)= 
\max_{q \in \Delta_n} ~\left\{\sum_{i=1}^n q_i \,\ell(y_i, \varphi(w,x_i))  : 0\leq q_i \leq \frac{1}{n(1-p)}\right\}
\end{equation}
where $\Delta_n$ denotes the probability simplex 
$\Delta_n =\{ q\in (\RR_+)^n, \sum^n_{i=1} q_i = 1\} $
Interestingly, this third formulation uncovers another interpretation of the superquantile objective.
The set of admissible probability $q_i$ in \eqref{eq:fminmax} acts as a so-called ambiguity set around the uniform probability distribution $(\frac{1}{n},\ldots,\frac{1}{n})$, relating \eqref{eq:ESM} to an instance of distributionally robust optimization: 
\eqref{eq:fminmax} considers the worst possible combination among possible re-weightings of the individual losses, with the probability distributions 
in this ambiguity set.
Superquantile-based learning is then expected to produce models that perform better in case of small distributional re-weighting between the training time and the testing time, compared to models trained using standard empirical risk minimization. 

The three above formulations \eqref{eq:ESM}
\eqref{eq:fminmin} and
\eqref{eq:fminmax} of the superquantile-based objective reveal an inherent non-smoothness. We discuss in the next section how to obtain first-order information from a superquantile-based criterion. Note, though, that training with such loss is not straightforward: replacing the expectation by the superquantile in \eqref{eq:ESM} completely changes the situation, making stochastic gradient algorithms, popular methods for solving~\eqref{eq:ERM}, which are somewhat flexible to the smoothness properties of the objective, not directly applicable; we will come back to this in Section~\ref{sec:algos}. 

Let us finally mention that the probability threshold $p$ should be considered as an hyperparameter of the superquantile-based learning problem~\eqref{eq:ESM}. The standard way to set $p$ is then to perform a cross validation over a grid of values and chose the best one with respect to a risk sensitive metric, such as e.g. the $90^{th}$ percentile of the validation loss.

{We finish this section by illustrating on a toy problem}
that super\-quantile-based learning allows one, as expected, to learn models with better worst-case performance.

\begin{example}[Superquantile-based learning gives better worst-case performance]\label{ex:num}
We consider a linear regression task on a synthetic training dataset\footnote{We take $n=10^4$ points in $\mathbb{R}^{40} \times \mathbb{R}$. The design matrix $X = (x_i)_{1\leq i\leq n}$ is generated with the \texttt{make\_low\_rank\_matrix} procedure of \textit{scikit\_learn}\;\cite{scikit-learn} with a rank $30$.} to provide a striking illustration of the benefit of superquantile-based learning in terms of worst-case performance. For a given model parameter $\bar w$, we generate the data\;according\;to
\[
 y_i = x_i^\top \bar w + \varepsilon_i
 \qquad\text{with }\varepsilon_i = \beta\varepsilon_{\mathcal{N}} +  (1-\beta)\varepsilon_{\mathcal{L}}.
\]
The noise $\varepsilon_i$ 
is generated from a mixture of two distributions:
$\varepsilon_{\mathcal{N}}$ follows a standard normal distribution, 
$\varepsilon_{\mathcal{L}}$ follows a Laplace distribution with location $\mu=10$ and scale $s=1$, 
and $\beta$ follows a Bernoulli distribution with parameter\;$0.8$. 
We solve the ordinary $\ell_2^2$-regularized least squares problem
and its superquantile counterpart:
\[
    \min_{w\in \mathbb{R}^d} \mathbb{E}_{(x, y) \sim \mathcal{D}}\big((y - w^\top x)^2\big)
    \quad\text{vs.}\quad
    \min_{w\in \mathbb{R}^d} [\bar Q_p]_{(x, y) \sim \mathcal{D}}\big((y - w^\top x)^2\big).
\]
Figure\;\ref{fig:synthetic_dataset} reports 
the distribution of losses obtained on the training dataset
and on a test dataset of $2000$ data points independently generated with the same procedure.
Thanks to the superquantile-based learning,
the upper tail of the error is shift to the left of the plot, 
which in other words means an improved performance in extreme cases. \qed

\end{example}

\begin{figure}[h!]
\vspace{-1mm}
\begin{minipage}[b]{1.0\linewidth}
  \centering
  \centerline{\includegraphics[width=9cm]{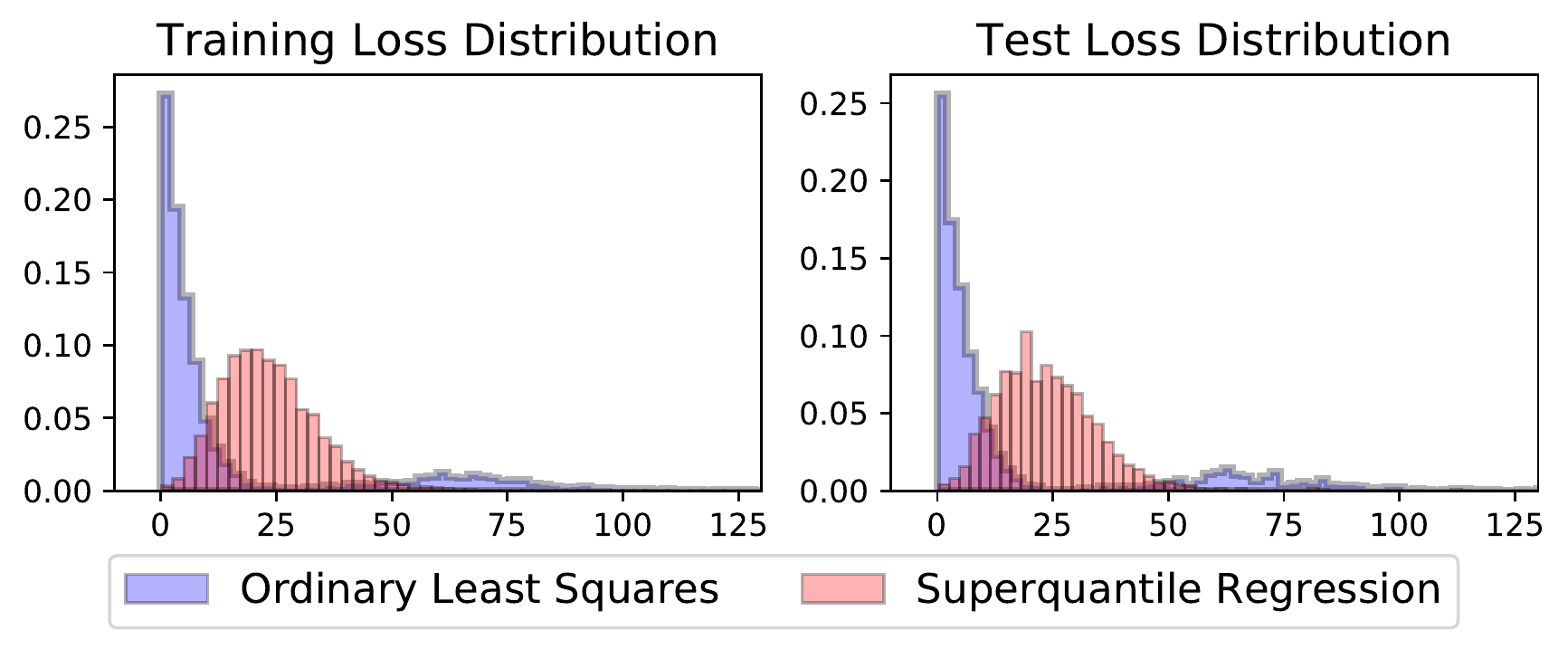}}
\end{minipage}
\vspace{-3mm}
\caption{Illustration of the reshaping of the distribution of errors resulting from superquantile-based learning (model trained with $p=0.9$).}
\label{fig:synthetic_dataset}
\end{figure}


\section{First-order oracles for superquantile function}
\label{sec:oracles}

The expression~\eqref{eq:integral_discrete} 
gives an efficient way to compute superquantiles. We have indeed a three step procedure: (i) compute the $p$-quantile with the specialized algorithm (called \texttt{quickfind}) of complexity $O(n)$ (with $n$ the number of data points); (ii) select all values greater or equal than the quantile;
(iii) average values along \eqref{eq:integral_discrete}. To minimize the superquantile-based objective~\eqref{eq:ESM}, we would also need, in addition to an objective evaluation oracle, an oracle to obtain first-order information.

In this section, we study the differentiability properties of the superquantile objective and we describe how to obtain subgradient or gradient information with the same complexity~$O(n)$ as for computing a standard quantile. We denote by 
\begin{equation}\label{eq:Li}
L^i(w) =\ell(y_i, \varphi(w,x_i))
\end{equation}
the underlying data-dependent functions 
in \eqref{eq:ERM} and \eqref{eq:ESM}. We will distinguish two cases: (a) $L^i$ convex in Section\;\ref{sec:subgrad} and (b) $L^i$ smooth in Section\;\ref{sec:grad}.

\subsection{Subgradient oracle}\label{sec:subgrad}

We assume here that the functions $L^i$ defined in \eqref{eq:Li} are convex. This is the case
when e.g.~the model $\varphi$ is linear and the loss $\ell$ is convex with respect to its second variable, as for the $l_2$-squared or the cross-entropy loss. This encompasses several situations including $p$-least-squares regression with $p\geq1$
     \[
         L^i(w) = |y_i - w^\top x_i|^p
     \]
or the logistic regression which can be written with $\hat{y}_i = 1/(1+e^{-\w^\top x_i})$ as
     \[
    L^i(w) = - y_i \log(\hat{y}_i) - (1-y_i) \log(1-\hat{y}_i). 
    \]

In this case, 
the superquantile-based function $f$ of\;\eqref{eq:ESM} is convex as well: we can see it on\;\eqref{eq:fminmax} which expresses $f$ is a max, over $q$, of convex functions in $w$. 
We give here the expression of the entire subdifferential for the convex case.
This result is not new: it is mentioned in several recent papers including\;\cite{levy2020large,curi2020adaptive}; it is part of the thorough study of~\cite{ruszczynski2006optimization} where gradients\footnote{Interestingly, the nonsmoothness of superquantile-based functions arises only with discrete distributions, as we consider here.} 
for general distributions are obtained from advanced tools. We give here a short proof using elementary convex analysis\;\cite{hiriart2013convex}. 

\begin{proposition}\label{thm:sub}
Assume 
that the $L^i$ are convex. 
Fix $\x \in \mathbb{R}^d$, compute $L(\w)\in\RR^n$ and $Q_p(L(\x))\in \RR$. 
Consider $I_>$ the set of indices such that $L^i(\x) > Q_p(L(\x))$ and $I_=$ the set of indices such that $L^i(\x) = Q_p(L(\x))$. 
Then the subdifferential at $\x$ of the convex function $f$ reads as the 
Minkowski sum
\begin{equation}\label{eq:sub}
\partial f(\x)  
~=~ \frac{1}{n(1-p)}\!\sum_{i\in I_>} \partial L^i(\x) 
~+~ \frac{\delta}{1-p} \conv\left\{ \partial L^i(\x) : i \in I_=\right\},
\end{equation}
with $\delta = \frac{1}{n} (n - |I_>|) - p$.
In particular, when the $L^i$ is differentiable at $\x$, $f$ is differentiable at\;$\x$ if and only if the set  $I_=$ is reduced to a singleton. 
\end{proposition}

\begin{proof}
The proof simply consists in applying convex calculus rules; the reader may find them in ~\cite[Chap\;D]{hiriart2013convex}. First 
we apply
Theorems\;4.1.1 and\;4.4.2 
to $h_i(\x, \eta)=\max\{L^i(\x) - \eta, 0\})$ to get
\begin{equation*}
    \partial h_i(x, \eta) = \{(\partial L^i(\x), -1) (\mathbbm{1}_{L^i(\x) > \eta} + \alpha \mathbbm{1}_{L^i(\x) = \eta}),\; \alpha \in [0, 1]\}
\end{equation*}
We apply
Theorem\;4.1.1 
with $h(\x, \eta) =\eta + \frac{1}{(1-p)n} \sum_{i=1}^n h_i(\x, \eta)$
\begin{equation*}
\begin{split}
    \partial h(\x, \eta) =  &\left\{ \left(\frac{1}{(1-p)n} \sum_{i=1}^n \partial L^i(\x) \delta^i(w,\alpha),
    \right. \right.\\
    &  \left. \left. 1 - \frac{1}{(1-p)n} \sum_{i=1}^n \delta^i(w,\alpha) \right), \alpha_i \in [0,1], \forall i \right\}.
\end{split}
\end{equation*}
with $\delta^i(w,\alpha) = (\mathbbm{1}_{L^i(\x) >  Q_p(L(\x))} + \alpha_i \mathbbm{1}_{L^i(\x) =  Q_p(L(\x))})$.
We finish with writing $f(\x)\!=\!\min_{\eta \in \mathbb{R}} h(\x, \eta)$ from \eqref{eq:fminmin}.
We can thus apply Corollary\;4.5.3 to get \eqref{eq:sub} after simplification.
\qed
\end{proof}

This proposition thus tells us that the computation of a subgradient can be performed in linear time from the subgradients $g_i\in \partial L^i(\x)$ for $i\in I_>\cup I_=$: the computing cost essentially stems from the computation of the $p$-quantile of the losses $L^{i}(\x)$ and the sum of vectors in $\RR^d$.

\subsection{Gradient oracle (for smoothed approximation)}\label{sec:grad}

We assume in this section that the functions $L^i$ defined by \eqref{eq:Li} are smooth, which holds locally when both the model $\varphi$ and the loss $\ell$ are smooth. Unfortunately, the superquantile breaks the smoothness (see e.g. Proposition~\ref{thm:sub} with smooth convex functions $L^i$), so that superquantile-based function $f$ is usually nonsmooth.

We propose here to smooth $f$ using infimal convolution as in~\cite{nesterov2005smooth}. More precisely, we follow the methodology of~\cite{doi:10.1137/100818327} and we propose to smooth only the superquantile $\bar Q_p$ rather than the whole function\;$f$. Given the formulation \eqref{eq:fminmax}, we consider the function $f_\mu$ for $\mu>0$, as the composition of the $L^i$ by the infimal convolution smoothing of $\bar Q_p$ 
\begin{equation}\label{eq:def_f_mu}
    f_{\mu}(\x) = \max_{q \in \Delta_n, q_i\leq \tfrac{1}{n(1-p)}} \sum_{i=1}^n q_i \; L^i(\x)
    - \mu\ d(q) \qquad\text{} 
\end{equation}
where $d:\mathbb{R}^n \rightarrow \mathbb{R}$ is a fixed non-negative strongly convex function. 
As a direct application of~\cite[Th.\;1]{nesterov2005smooth}, we have the following proposition establishing that $f_\mu$ is a smooth approximation of $f$.

\begin{proposition}[Gradient of smoothed approximation]\label{thm:smooth_approx}
Assume that 
the $L^i$ are smooth for any $i$. 
Then, the function $f_\mu$ of \eqref{eq:def_f_mu} provides a global approximation of $f$, i.e.
$f_\mu(\x) \leq f(\x) \leq  f_\mu(\x) + \frac{\mu}{2}$ for any $\x \in \mathbb{R}^d$.
If $L$ is differentiable, then $f_\mu$ is differentiable as well, with 
\begin{equation}\label{eq:grad_phi_mu}
\nabla f_\mu(\x) = \jac(\x)^T q_\mu(\x),
\end{equation}
where $\jac(\x)$ is the Jacobian of $L$ at $\x$ and $q_\mu(\x)$ is the optimal solution of~\eqref{eq:def_f_mu}, unique by
strong convexity of $d$.
\end{proposition}

In practice, the previous result requires an efficient subroutine solving~\eqref{eq:def_f_mu}.
Here, we consider the euclidean distance to the uniform probability measure
\begin{equation}\label{eq:div}
d(q) = \sum_{i=1}^n \big(q_i - \tfrac{1}{n}\big)^2.
\end{equation}
For this distance, Algorithm\;\ref{algo_projection} provides an efficient procedure for solving \eqref{eq:def_f_mu}.
The procedure follows the one in~\cite{condat2016fast}, where convex duality and one-dimensional search ideas are fruitfully combined. 
Thanks to the particular smoothing distance $d$, computing \eqref{eq:def_f_mu} by duality is equivalent to finding the zero of a non-decreasing, piecewise affine and continuous function (the derivative of the dual function), which has an explicit expression after sorting the kinky points. We formalize this in Algorithm \ref{algo_projection} and the proposition below. 

\begin{algorithm}[t!]
\KwInit{$e=(1, \dots,1)^\top$, ~~$u= L(\x) + \frac{\mu}{n} \e$, ~~$\ell = \frac{1}{n(1-p)}$,~~$q_\mu = 0\in \mathbb{R}^n$ 
}
\nosemic Find in the points of non-differentiability $\mathcal{P}$, $a$ and $b$ such that, \;
\pushline\dosemic\nonl  $\mathcal{P}:= \{u_i, u_i -  \mu \ell, i \in \{1,\dots,n\}\}$ \\
\dosemic\nonl $a := \max \left\{s \in \mathcal{P},\theta'(s) \leq 0 \right\}$ \\
\dosemic\nonl $b := \min \left\{s \in \mathcal{P}, \theta'(s) > 0 \right\}$ \;
Set the dual optimal solution $\lambda := a - \frac{\theta'(a)(b-a)}{\theta'(b) - \theta'(a)}$\;
    %
Construct the primal solution component-wise: \\
\For{$1 \leq k \leq n$}{
  \uIf{$\lambda < u_k - \mu \ell$}{
    ${[q_\mu]}_k = \ell$\;
  }
  \uElseIf{$u_k - \mu \ell \leq \lambda < u_k$} {
    ${[q_\mu]}_k = \frac{u_k - \lambda}{\mu}$\;
  }
  \Else{
    ${[q_\mu]}_k = 0$
  }
  }
\KwOut{$q_\mu \in \mathbb{R}^n$ : solution of \eqref{eq:def_f_mu}}
  
\caption{Fast subroutine for smoothed oracle}
\label{algo_projection}
\end{algorithm}

\begin{proposition}\label{thm:On}
    Algorithm \ref{algo_projection} computes the optimal solution of the problem \eqref{eq:def_f_mu} with $d$ as \eqref{eq:div} at a cost of $O(n)$ operations.
\end{proposition}

\begin{proof}
We dualize the constraint $\sum_{{i=1}}^n q_i - 1 = 0$ to get the Lagrangian
\begin{equation*}
\mathscr{L}(q,\lambda) = 
\sum_{{i=1}}^n q_i L^i(\x) - \frac{\mu}{2} \sum_{{i=1}}^n \left(q_i - \frac{1}{n}\right)^2 + \lambda \left(1 - \sum_{{i=1}}^n q_i\right).
\end{equation*}
With $\ell$ and $u$ introduced in the algorithm, the dual function writes:
\begin{equation*}
\theta(\lambda) = \max_{\substack{q \in \mathbb{R}^n \\ 0 \leq q_i \leq l}} \mathscr{L}(q,\lambda) =  \lambda - \frac{\mu}{2n} + \sum_{{i=1}}^n \max_{{0 \leq q_i \leq l}} (u_i - \lambda)q_i - \frac{\mu}{2} q_i^2 
\end{equation*}
For $\lambda \in \mathbb{R}$ and $i \in \{1,\dots, n\}$ fixed, let us introduce the function $h_i(q_i) = (u_i - \lambda)q_i - \frac{\mu}{2} q_i^2$. 
Then, we get
\begin{equation}\label{eq:sol_q_i}
\begin{split}
    \argmax_{{0 \leq q_i \leq l}}h_i(q_i) &=  \left\{
    \begin{array}{lll}
        0 &\mbox{ if } \lambda \geq u_i \\
        \frac{u_i - \lambda}{\mu} &\mbox{ if } u_i \geq \lambda \geq u_i - \mu \ell \\
        \ell  &\mbox{ if } \lambda \leq u_i - \mu \ell  \\
    \end{array}
    \right.        
\end{split}
\end{equation}
As a result, we get the explicit expression of $\theta(\lambda)$. Observing that it is differentiable, we get
\begin{equation*}
    \theta'(\lambda) = 1 - \sum_{{i=1}}^n \left(\frac{u_i - \lambda}{\mu} \mathbbm{1}_{u_i \geq \lambda \geq u_i - \mu \ell} + \ell \mathbbm{1}_{u_i - \mu \ell > \lambda}\right).
\end{equation*}
Observe that $\lim_{{\lambda \rightarrow + \infty}} \theta'(\lambda) = 1$ and since $n\ell =\frac{1}{1-p} > 1$,  $\lim_{{\lambda \rightarrow - \infty}} \theta'(\lambda) < 0$. Therefore, $\theta'$ is a non-decreasing and continuous (piecewise affine) function that takes negative and positive values: by the intermediate value theorem, there exists a solution $\lambda^\star \in \mathbb{R}$ such that $\theta'(\lambda^\star) = 0$. By duality theory, the associated $q^\star$ (the optimal solution of\;\eqref{eq:sol_q_i} for $\lambda= \lambda^\star$) is the solution of the primal problem\;\eqref{eq:def_f_mu}.
Finally, we compute $\lambda^\star$ zeroing $\theta'$. Since $\theta'$ is piecewise affine, we just need to evaluate  $\theta'$ at points belonging to the set $\mathcal{P}$ and at $a$ and $b$ as defined in Algorithm\;\ref{algo_projection}. 
Thus we have $\lambda^* = a - \frac{\theta'(a)(b-a)}{\theta'(b) - \theta'(a)}$.
Regarding computational costs, this algorithm boils down to the search of $a$ and $b$, and the assignment of the coordinates of $q_\mu$. This also sums up to a $\mathcal{O}(n)$ cost.\qed
\end{proof}

%

Combining Propositions\;\ref{thm:smooth_approx} and\;\ref{thm:On} provides a gradient oracle for the smoothed approximation $f_\mu$. For a given $w\in \RR^d$, we run Algorithm~\ref{algo_projection} to get $q_\mu(w)$; we select the indexes $i$ of non-zeros entries of $q_\mu(w)$; 
and from the oracles of $L^i$ we get
\[
    f_\mu(w) = \!\!\!\sum_{i: {q_\mu(w)}_i \!\neq 0} \!\!\big(q_\mu(w)\big)_{\!i} \,L^i(w)~~~\text{and}~~~ 
    \nabla f_\mu (w) = \!\!\!\sum_{i: {q_\mu(w)}_i \!\neq 0}\!\! \big(q_\mu(w)\big)_{\!i} \nabla L^i(w).
\]

{We finish this section 
by a short discussion on the two extreme cases for the smoothing parameters : $\mu$ close to $0$ and $\mu$ very large. Small $\mu\sim0$ imply exploding entries of $q_\mu(w)$ (see line 8 in Algorithm\;\ref{algo_projection}) and then instability of $\nabla f_\mu (w)$. Large $\mu\sim+\infty$ imply 
$q_\mu(w) = (1/n,\ldots,1/n)$ constant (see line 6 in Algorithm\;\ref{algo_projection}) and therefore the smoothed function $f_\mu (w)$ and its gradient $\nabla f_\mu (w)$ coincide with the function and gradient of the corresponding ERM objective. We illustrate these two extreme cases in Section\;\ref{sec:numexp}.}

\section{First-order optimization for superquantile-based learning}
\label{sec:algos}

Minimization of superquantile-based objectives comes with a number of technical challenges on the structure of the problem tackled, the size of the dataset or the non-smoothness of the objective. 
Standard works on minimizing superquantiles considered linear programming
or convex programming techniques, including interior point algorithms; see the review of\;\cite{rockafellar2014superquantile}. Perhaps surprisingly, the use of first-order algorithms for superquantile-based optimization is quite recent and seems to have been driven by domain applications of machine learning.

In this section, we provide an overview of the range of first-order methods to minimize superquantile-based objective functions expressed as \eqref{eq:ESM}, \eqref{eq:fminmin}, or\;\eqref{eq:fminmax}. Our discussion focuses on practical considerations; we give pointers to references presenting more details and theoretical analysis {(in particular, convergence results and convergence rates if any)}. 

\subsection{Batch algorithms}

As explained in Section\;\ref{sec:oracles}, computing the function values and (sub)gradients of the superquantile-based function $f$ in \eqref{eq:ESM} (or its smoothed counterpart $f_\mu$) requires sorting loss values on the whole data set, which is not directly amenable to classical stochastic gradient algorithms. This rehabilitates batch optimization algorithms, at least for small to medium datasets. Thus the first approach for minimizing the superquantile-based objective functions is to use standard subgradient-based methods (subgradient and dual averaging) or gradient-based methods (gradient, accelerated gradient, Quasi-Newton). This is essentially what we described in~\cite{laguel2020first}, and it is the first set of methods available in our toolbox. More precisely, we have two cases:
\begin{itemize}
    \item \textit{Convex case.} If $w \mapsto \ell(y_i, \varphi(w,x_i))$ are convex, then $f$ is convex and we have a subgradient oracle (from Proposition\;\ref{thm:sub}) enjoying the same complexity as the one for computing a quantile. We can use standard convex nonsmooth optimization methods, such as subgradient methods and dual averaging.  We implement in particular the ``weighted'' version of the latter with a Euclidean prox-function~\cite[Eq.\;2.22]{nesterov2009primal}. These algorithms satisfy ergodic convergence guarantees in objective values~\cite{bertsekas2015convex}. 
    

    \medskip
    \item \textit{Smooth case.} If $w \mapsto \ell(y_i, \varphi(w,x_i))$ are differentiable, then 
    we have a gradient oracle of the smooth approximation $f_\mu$ (from Proposition\;\ref{thm:On}), again with a $\mathcal{O}(n)$ complexity. We can use standard methods for smooth optimization: gradient method, accelerated gradient method, and quasi-Newton\;(L-BFGS). If furthermore we have convexity, these algorithms satisfy convergence guarantees in objective values~\cite{bertsekas2015convex,bertsekas2016nonlinear}.

\end{itemize}

For small to medium-size dataset, such batch methods are shown to be simple and efficient; see forthcoming Section~\ref{sec:exp1}. For large-scale problems though, the oracles become too costly as they require sorting loss values on the whole data set. We turn to the other formulations to introduce stochastic and mini-batch algorithms, that usually are the methods of choice for the case of standard learning using empirical risk minimization.

\subsection{Mini-batch algorithms}

From the perspective of the formulation \eqref{eq:fminmin} of the objective, the superquantile-based learning problem writes 
\begin{equation}\label{eq:minmin}
\min_{w \in \Rd}~
\min_{\eta \in \RR} ~\left\{ \frac{1}{n(1-p)} \sum_{i=1}^n \max\{\ell(y_i, \varphi(w,x_i)) - \eta , 0\} + \eta \right\}.
\end{equation}
When the loss is assumed to be smooth, one may again smooth the inner $\max\{\cdot,0\}$ term to get a smooth approximation of this joint objective.
One can then perform a joint minimization\footnote{Such approach is well-suited to problems with a particular decomposable structure such as non-anticipativity constraints in multi-stage programming problems; see~\cite{rockafellar2018solving}.}  with respect to the model $\x$ and the dual variable $\eta$. In other words, superquantile learning reduces to a standard empirical risk minimization
with a modified loss function truncated by the max-term. In practice, batch methods may not be interesting here, since they would not leverage the fact that the minimization over $\eta$ can be performed explicitly. 
Thus \cite{laguel:device} proposes, in a context of federated learning, to rather perform independent minimization over $w$ and $\eta$ alternatively. 
In general, this min-min approach \eqref{eq:minmin} paves the way to stochastic and mini-batch algorithms. 

Several works, including \cite{soma2020statistical} and \cite{willaimson2019fairness} (as well as \cite{fan2017learning} without mentioning superquantile),
use successfully standard stochastic optimization algorithms on this modified objective. 
Observe though that, if a mini-batch of data is sampled uniformly at random from the data, only a fraction $(1-p)$ carry (sub)gradient information. 
Furthermore, the (sub)gradients of these examples 
are scaled by $\frac{1}{1-p}$, leading to exploding directions. Thus mini-batch estimates of (sub)gradients of superquantile-based objectives may suffer from high variance. 
A solution proposed by \cite{curi2020adaptive} is to perform an adaptive sampling rather than a uniform one. This algorithm 
gradually adjusts its sampling distribution to increasingly sample tail events, until it eventually minimizes the superquantile. This approach has a nice two-player interpretation related to the third formulation, recalled\;below.

The third expression \eqref{eq:fminmax} of $f$ leads to the following formulation 
(or, as previously, its smoothed counterpart with a quadratic term on $q$ as in \eqref{eq:def_f_mu})
\begin{equation}\label{eq:minmax}
\min_{w \in \Rd}~
\max_{q \in \Delta_n} ~\left\{\sum_{i=1}^n q_i \,\ell(y_i, \varphi(w,x_i))  : 0\leq q_i \leq \frac{1}{n(1-p)}\right\}.
\end{equation}
This min-max formulation offers several ways to solve the superquantile-based learning. A first approach would consist in considering it as a generic saddle point problem and using standard (extra-)gradient algorithms or recent extensions exploiting some aspects of the problem
(see e.g.\;\cite{luo2020stochastic} for a variance-reduced min-max with strongly concave max). In our specific case, computing the max can be done systematically by a greedy algorithm with linear time complexity (see Section\;\ref{sec:oracles}). This key feature is exploited by the stochastic algorithm of \cite{duchi2019variance}, and also by the one of \cite{kawaguchi2020ordered} without relating it to superquantile. This algorithm uses a biased sampling approximation to $f$ or $f_\mu$ which has nice 
guarantees. We briefly describe below this approach.

We sample a mini-batch of $\Dduchi$ uniformly in $\mathcal{D}$ and we consider the restriction 
\begin{eqnarray*}
\fduchi(w) &=& {[{\bar Q}_{p}]}_{(x,y)\sim \mathcal{S}}\big(\ell(y, \varphi(w,x)\big)\\
&=& \max_{q \in \Delta_n} ~\left\{\sum_{i\in \Dduchi} q_i \,\ell(y_i, \varphi(w,x_i))  : 0\leq q_i \leq \frac{1}{n(1-p)}\right\}.
\end{eqnarray*}
We can now use the (sub)gradient oracles of Section\;\ref{sec:oracles} on $\fduchi$ and apply gradient-based algorithms with biased mini-batch estimator. Indeed, even if
$\EE[\fduchi(w)] \neq f(w)$ and $\EE[\nabla \fduchi(w)] \neq \nabla f(w)$,
under standard assumptions, the bias is controlled by uniform bounds and variance bounds, which gives in turn complexity guarantees when using gradient-based algorithms; see \cite[Sec.\;3]{levy2020large}. The algorithm requires a number of gradient evaluations independent of training set size and number of parameters, making it suitable for large-scale applications. This algorithm is implemented in our toolbox {and tested in Section\;\ref{sec:numexp}.}



\section{\texttt{SPQR}: Python Toolbox for Superquantile-based learning}
\label{sec:spqr}
We provide a Python software package for superquantile-based learning; it is named \texttt{SPQR} for SuPer Quantile Risk optimization.  
The software package includes modeling tools and optimization algorithms to solve problems of the form\;\eqref{eq:ESM} with just a few lines of code. The implementation builds off basic structures of~\texttt{scikit-learn} \cite{scikit-learn} the popular python machine learning library. 
\texttt{SPQR} routines rely on just-in-time compilation~\cite{10.1145/2833157.2833162} to ensure efficient running times. The software package is publicly available at~\url{https://github.com/yassine-laguel/spqr}. We now walk the reader through the toolbox~\texttt{SPQR}. 


\subsection{Basic usage: input format and execution}
   
The user provides a dataset modeled as a couple
$(X,Y)\in \mathbb{R}^{n\times p}\times\mathbb{R}^p$ and 
a first-order oracle for the function $L^{i}$.
The dataset is stored into two numpy arrays \texttt{X} and \texttt{Y}; for instance, 
for realizations of random variables:
\begin{lstlisting}
    import numpy as np
    X = np.random.rand(100, 2)
    alpha = np.array([1., 2.])
    Y = np.dot(X, alpha) + np.random.rand(100)
\end{lstlisting}
The two python functions \texttt{L} and \texttt{L\_prime} are assumed to be functions of the triplet \texttt{(w,x,y)} where \texttt{w} is the variable and \texttt{(x,y)} a datapoint. For instance, the oracle for superquantile linear regression are the following one.
\begin{lstlisting}
    # Define the loss and its derivative
    def L(w,x,y):
        return 0.5 * np.linalg.norm(y - np.dot(x,w))**2
    def L_prime(w,x,y):
        return -1.0 * (y - np.dot(x,w)) * x 
\end{lstlisting}
Before solving\;\eqref{eq:ESM}, we instantiate the \texttt{RiskOptimizer} object 
with the 
oracles, following the standard usage of \texttt{scikit-learn}. The basic instantiation is:
\begin{lstlisting}
    from SPQR import RiskOptimizer
    # Instantiate a risk optimizer object
    optimizer = RiskOptimizer(L, L_prime)
\end{lstlisting}
\texttt{RiskOptimizer} inherits from \texttt{scikit-learn}'s estimators: we use the \texttt{fit} method to run the optimization algorithm on the provided 
data, 
to get a solution of\;\eqref{eq:ESM}.
\begin{lstlisting}
    #  Running the algorithm
    optimizer.fit(X,Y)
    sol = optimizer.solution 
\end{lstlisting}

\vspace*{-1ex}

\subsection{Advanced use: parameters and \texttt{SPQR} objects}


\paragraph{Options and parameters.} The customizable parameters are stored in a python dictionary \texttt{params} which is designed as an attribute of the\;\texttt{RiskOptimizer} class. The main parameters to tune are: the choice of the oracle, the choice of the algorithm, the safety probability level \texttt{p}, the starting point of the algorithm \texttt{w\_start}, the maximum number of iterations \texttt{max\_iter}. The user can specify some of these parameters as an input and the others will be filled with defaults values when instantiating a \texttt{RiskOptimizer}. For example:
\begin{lstlisting}
    custom_params = {'algorithm': 'dualaveraging',  # selected algorithm
        'p': 0.2 }  # safety probability level
    custom_optimizer = RiskOptimizer(loss, loss_prime, params=custom_params)
\end{lstlisting}
Some important parameters (such as the safety probability level, the algorithm chosen, or the smoothing parameter $\mu$) can be given directly to the constructor of the class \texttt{RiskOptimizer} when instantiating the object. For example:
\begin{lstlisting}
    other_custom_optimizer = RiskOptimizer(loss, loss_prime, p=0.95, algorithm='bfgs', mu=0.1)
\end{lstlisting}    

\paragraph{Oracle classes.} 
The selection of the oracle is automatically done when the user instantiate the \texttt{RiskOptimizer} object. 
{Four different oracles are implemented as python objects:
two oracles for batch methods (\texttt{OracleSubgradient} to be used when the chosen algorithm is \texttt{'subgradient'} or \texttt{'dual\_averaging} and \texttt{OracleSmoothGradient} when the chosen algorithm is \texttt{'gradient'}, \texttt{'nesterov'} or \texttt{'bfgs'}) and two mini-batch oracles 
(\texttt{OracleStochasticSubgradient} and \texttt{OracleStochasticGradient}).}

To avoid the treatment of optional parameters when instantiating an oracle, we advise to go through the instantiation of a \texttt{RiskOptimizer} first.
\begin{lstlisting}
    custom_params = {'algorithm': 'nesterov',  # selected algorithm
        'p': 0.5  # safety probability level
    }
    #  Instantiation of the Risk Optimizer
    custom_optimizer = RiskOptimizer(loss, loss_prime, params=custom_params)
    #  Recovery of the oracle
    smooth_oracle = custom_optimizer.oracle
\end{lstlisting}

\vspace*{-1ex}
\paragraph{Algorithms class}

The algorithm chosen is a parameter for the instantiation of the \texttt{RiskOptimizer} class. This parameter can either be given in the input dictionary \texttt{params} or directly to the constructor of \texttt{RiskOptimizer}. The user has the choice among \texttt{'subgradient'}, \texttt{'dual\_averaging}, \texttt{'gradient'}, \texttt{'nesterov'}, \texttt{'bfgs'} and \texttt{'sgd'}.

\begin{lstlisting}
    #  Risk Optimizer class with nesterov accelerated gradient algorithm
    custom_optimizer = RiskOptimizer(loss, loss_prime, algorithm='nesterov')
\end{lstlisting}
Each algorithm is implemented as a python class that stores the oracle, 
together with relevant parameters for the optimization process. 
The main method of this class is \texttt{run}, which is 
run when \texttt{RiskOptimizer.fit} is called. The parameters of the algorithm selected are stored in the dictionary \texttt{params} that is an input of the class \texttt{RiskOptimizer}. Hence, in a standard usage, there is no need to interact with the algorithm python object.

\section{Numerical experiments}\label{sec:numexp}
In this section, we report two types of numerical experiments:
\begin{itemize}
    \item ``Optimization" experiments in Section\;\ref{sec:exp1}. There are many algorithmic options within the toolbox \texttt{SPQR}; we provide a comparison of batch vs.\;mini-batch algorithms and a discussion on tuning the smoothness parameter. 
    \item ``Learning" experiments in Section\;\ref{sec:exp2}. The interest of using superquantile in learning has been shown empirically in several recent papers, including \cite{willaimson2019fairness,laguel:device,levy2020large,curi2020adaptive,soma2020statistical}. We provide here complementary experiments highlighting the robustness of superquantile-learnt models.
\end{itemize}
All experiments are run using~\texttt{SPQR}. The optimization algorithms are initialized at $w=0\in \Rd$. For these experiments, we use a bunch of standard datasets from the UCI repository, which scale from $352$ to $94644$ datapoints. For each dataset, categorical features were one-hot encoded so that the total number of features ranges from $3$ to $287$. 

For the experiment of Table\;\ref{tab:shift_figures}, we report the agreggated results for all the datasets. For the other experiments, we report, in the main text, the detailed results obtained with one representative dataset, and we provide, in appendix, complementary results for others datasets.

\begin{figure}[h!]
\centering
  \centerline{\includegraphics[width=5.8cm]{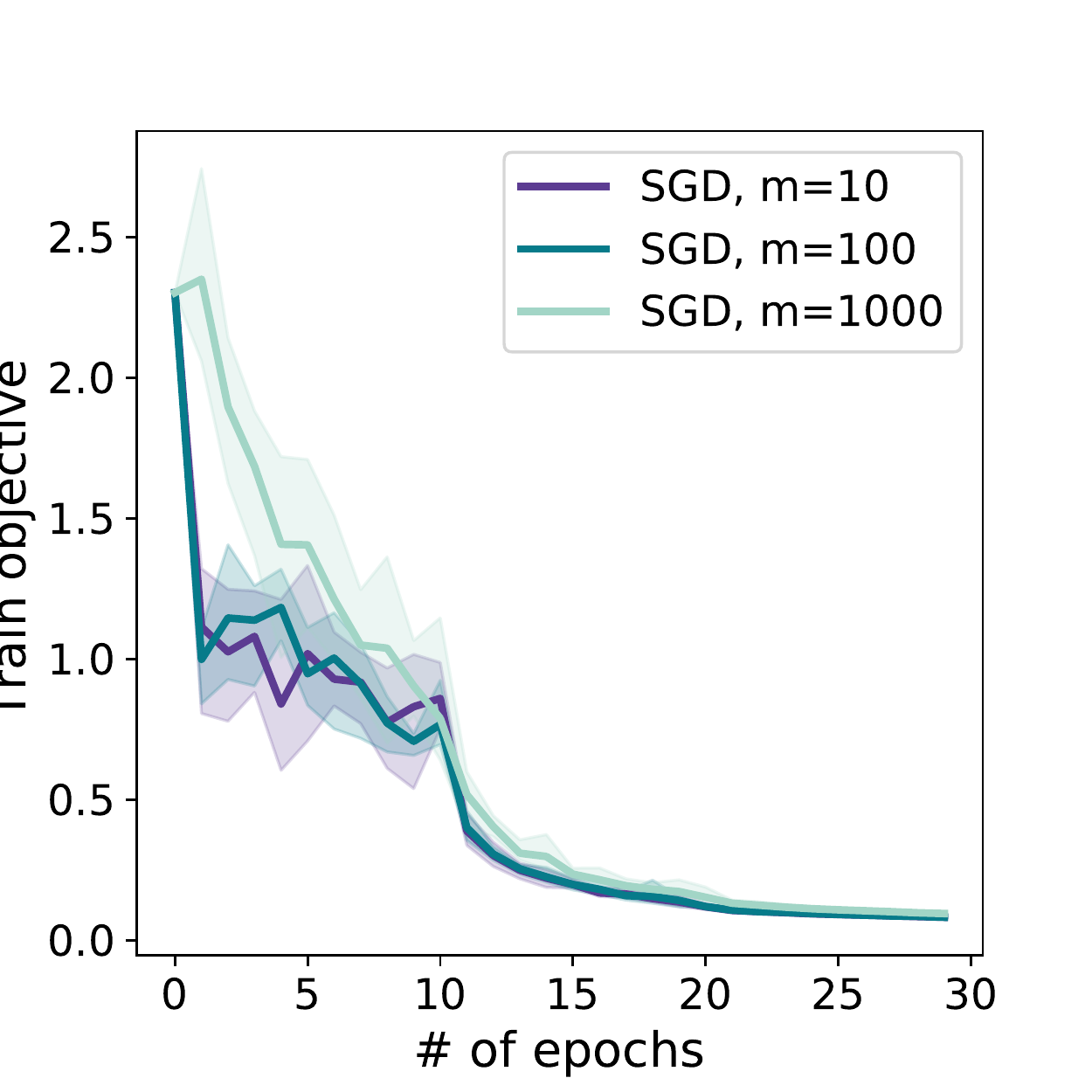}~~~\includegraphics[width=5.8cm]{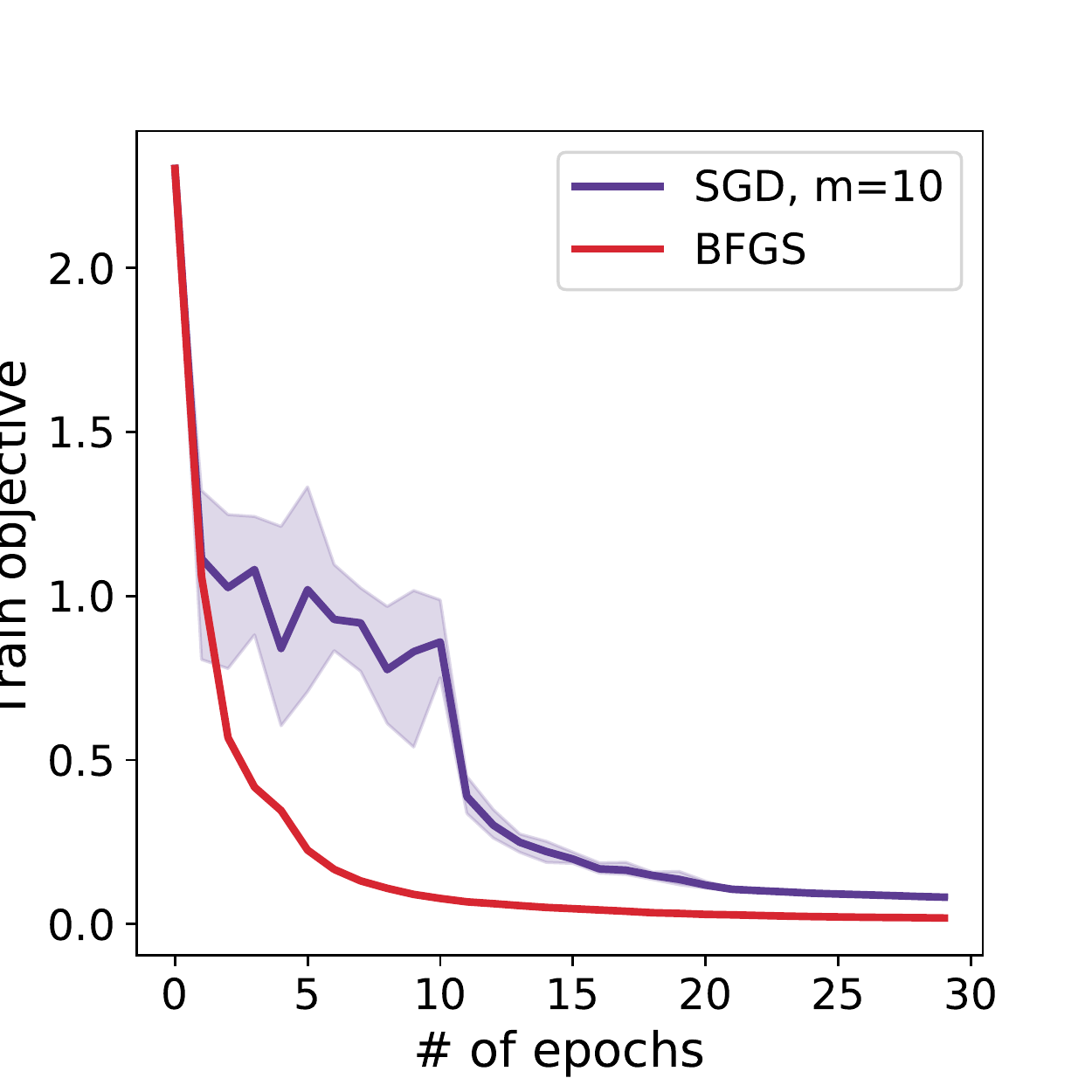}}
\caption{A comparison between batch/mini-batch algorithms in \texttt{SPQR} on a superquantile logistic regression problem with \texttt{MNIST}. Left: comparison of the runs of SGD with different batch sizes. Right: best SGD vs.\;batch quasi-Newton.
\vspace*{0ex}}
\label{fig:delta}
\end{figure}

\subsection{Solving superquantile-based learning.}\label{sec:exp1}

In this section, we illustrate two different aspects of the optimization methods available in \texttt{SPQR}. 
First, we compare the two families of algorithms available (batch vs.\;mini-batch)
showing the interest of using batch algorithms for superquantile-based learning within \texttt{scikitlearn/SPQR}. Second, 
we experiment with all the range of the smoothing parameter, advocating to avoid extreme values.

\paragraph{Batch vs.\;mini-batch.}
We compare, on a standard problem, a stochastic gradient algorithm (more precisely, SGD with momentum, denoted SGD) and a batch quasi-Newton algorithm (more precisely, low-memory BFGS \cite{nocedal2006numerical}, denoted BFGS). 

For this experiment, the set-up is similar to the one of~\cite{levy2020large}. 
We consider a supervised multi-class classification task with the superquantile multinomial logistic loss on the \texttt{MNIST} dataset. We perform feature extraction from the images using a pre-trained convolutional network similarly to~\cite{levy2020large}. For a fixed probability threshold set to $p=0.8$, we then train a linear multi-class classifier on top of the transformed data. 
For SGD, we use a momentum term of $0.9$ and we use a step decay scheme $\eta_t = \eta_0 d^{-\lfloor t/t_0 \rfloor}$, where $\eta_0$ is tuned with respect to the size of the mini-batch $m$, and where $d=0.5$ and $t_0=10$ epochs are fixed throughout all the experiments. 
For each mini-batch size $m \in \{10, 100, 1000\}$, we tune $\eta_0$ via a grid-search and take the highest initial value yielding a non-diverging sequence of iterates. In constrast with SGD, the quasi-Newton algorithm does not require specific tuning as it automatically calibrates stepsizes by line-searches at each iteration.

{On the left part of Figure~\ref{fig:delta}, we compare the performance of SGD for the different mini-batch sizes. Each color corresponds to a mini batch size $m \in \{10, 100, 1000\}$. Along iterates, the bold line represents the mean value over the five seeds of the functions and the shaded region represent the difference between the min and max values across the seeds.
We observe that there is no substantial difference among the sizes of the mini-batches: all curves show a noisy behaviour (caused by the stochastic approximation of the gradient at each step) and eventually converge to a suboptimal value. Unlike SGD, L-BFGS (right part of Figure~\ref{fig:delta}) presents a stable convergence. We observe also that a large number of epochs is necessary for SGD to catch up with BFGS for superquantile-based training. This is to be contrasted with the usually small number of epochs necessary for SGD to catch with BFGS for expectation-based training or ERM. Note that a final bias remains visible between the stochastic methods and the deterministic BFGS, as expected by the theory laid down in~\cite{levy2020large}.}



\paragraph{Impact of the smoothing parameter.} 

We consider a logistic regression on the \texttt{Australian Credit} dataset. For a sequence of smoothing parameters $\nu$ evenly spread on a log scale, we train $w_{\nu}^\star$ by solving the superquantile learning objective with L-BFGS and $p=.99$.


On Figure\;\ref{fig}, we report both the value of the smoothed $.99$-superquantile (purple) and the non-smoothed $.99$-superquantile (dashed green) at the $w_{\nu}^\star$. We also train the standard empirical risk minimizer $w^\star$ and we report both the average loss (solid black line) and the non-smoothed $.99$-superquantile loss (dashed black line) at\;$w^\star$.

For very small values of $\nu~(<10^{-3})$, we observe unsuccessful termination of the L-BFGS algorithm, due to the failure of the line-search. For medium values of $\nu~ (<1)$, the value of smooth superquantile-based function at $w_{\nu}^\star$ roughly coincides the non-smooth one. Finally for high values of $\nu~ (>10^3)$, we observe that the smooth superquantile tends to the same optimal function value of the empirical risk minimizer $w^\star$, as expected from 
Section\;\ref{sec:grad}.

	


\begin{figure}[!ht]
\centering
\vspace*{-1ex}
  \centerline{\includegraphics[width=10cm]{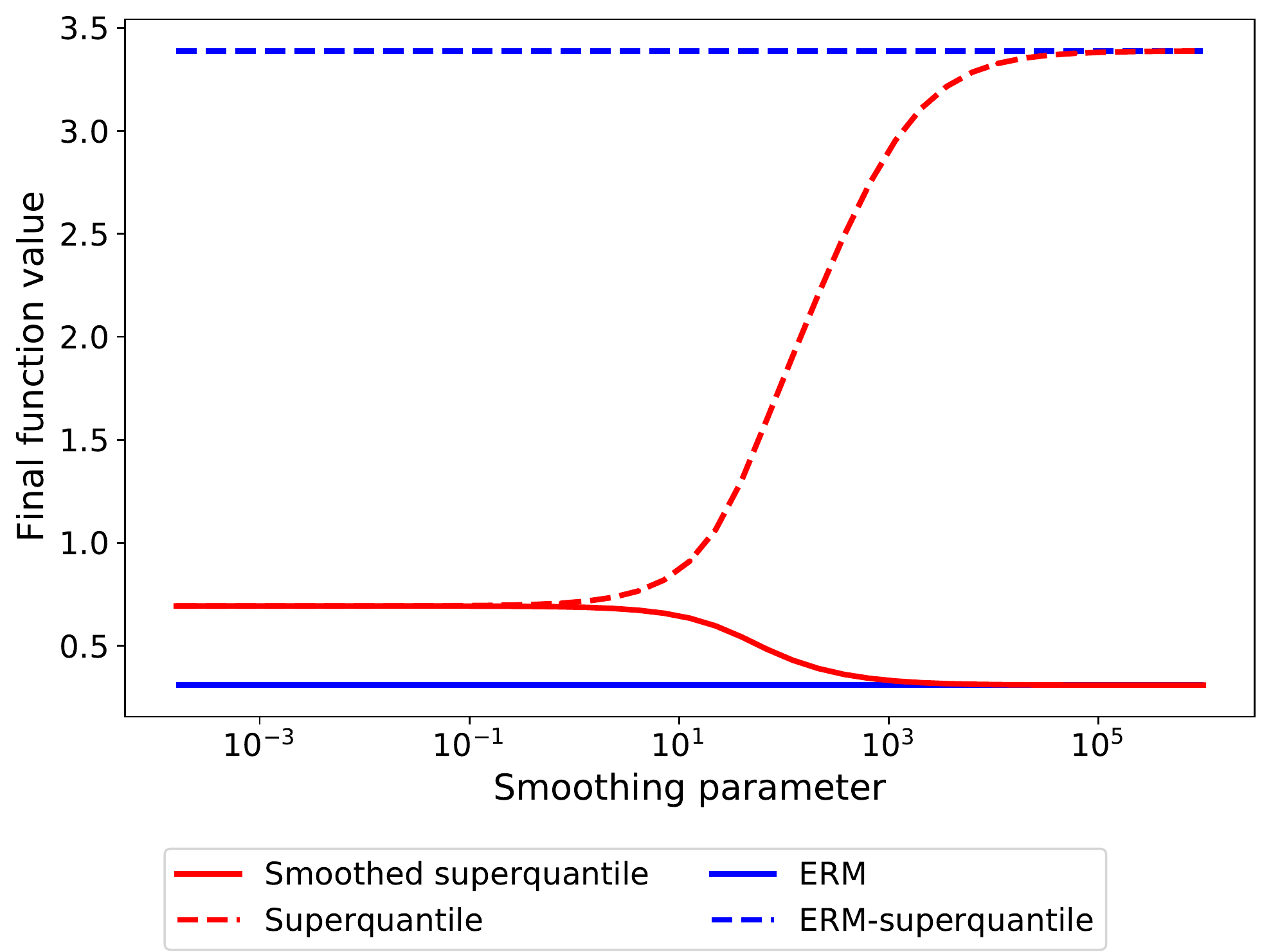}}
\vspace{-1ex}
\caption{Impact of the smoothing parameter $\nu$ on the results obtained by the quasi-Newton algorithm solving a superquantile logistic regression on the \texttt{Australian Credit} dataset. Medium values are preferable: small values compromise convergence and large values give solutions close to the standard ERM. 
\label{fig}}
\end{figure}

\subsection{Superquantile brings robustness against distributional shift}\label{sec:exp2}

In the second part of the numerical experiments, we show the benefits of the superquantile by comparing superquantile-based minimization vs. empirical risk minimization, when a distributional shift occurs, similarly to \cite{curi2020adaptive}.
For the three next standard regression or classification tasks, we proceed as follows. For each dataset, we first perform a $80\%$-$20\%$ train-test split. Second, we minimize with respect to the train set a regularized objective, both in expectation and with respect to the superquantile: 
\begin{equation}\label{eq:logistic_regression}
\begin{split}
    & \min_{\x \in \Rd} ~\mathbb{E}_{(x,y) \sim D_{\text{train}}} \ell(y,\x^\top x) + \frac{\lambda}{2} \|\x\|_2^2 \\
    & \min_{\x \in \Rd} ~[\bar{Q}_p]_{(x,y) \sim D_{\text{train}}} \ell(y,\x^\top x) + \frac{\lambda}{2} \|\x\|_2^2 \\
\end{split}
\end{equation}
We set the regularization parameter $\lambda$ to be the inverse of the number of training data-points: $\lambda = 1/n_{\text{train}}$. The above problems are solved with \texttt{SPQR} using L-BGFS. Then we perform three different types of distributional shifts on the testing set and we compare the behaviour of the superquantile-based models and the ERM models.  
We develop this approach in the next three settings.

\paragraph{Superquantile ridge regression.}


We consider a ridge regression problem, that is \eqref{eq:logistic_regression} with $\ell(y,{\x}^{\top}x) = (y-{\x}^{\top}x)^2$, on the dataset \texttt{Cpu-small}.
We minimize the two problems, first, in expectation and, second, 
with respect to the superquantile 
with several safety thresholds $p \in \{0.3, 0.5, 0.7, 0.8, 0.9, 0.95, 0.99\}$. 

We report in Figure~\ref{fig:regression} the histogram of losses on the test set and compare each trained superquantile model (in red) with the ERM model (in blue). We observe that as the probability threshold $p$ grows, the right tail distribution of losses on the test set gets shifted to the left. In particular, a dramatic decrease of the $90^{\textrm{th}}$ quantile of the losses can be observed. Thus superquantile learning allows us to reduce worst-case losses. This comes with the price of lower performances on the left tail distribution. 
This trade-off between gain on extreme cases and loss on average is typical of the impact of superquantiles. We observe a similar trade-off for other datasets; see Figures~\ref{fig:regression_appendix} in appendix.


\begin{figure}[t!]
    \centering
   \includegraphics[width=\linewidth]{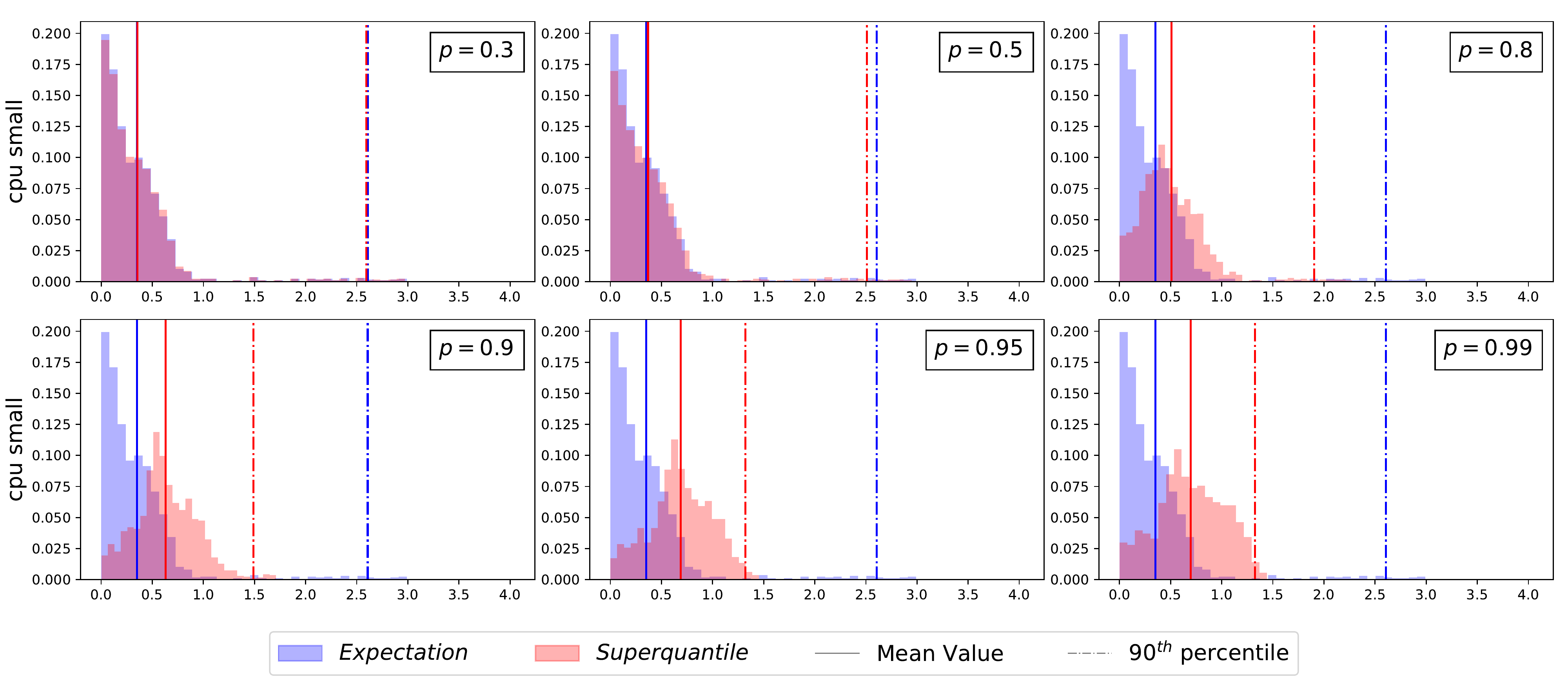}
\caption{Reshaping of the histogram of testing losses for superquantile regression models (in red) as $p$ grows. We observe a shift to the left of the $90^{\text{th}}$ quantile of losses, at the price of degrading the average value.\label{fig:regression}}
\end{figure}

\paragraph{Superquantile logistic regression.}


We consider a regularized logistic regression problem, that is\;\eqref{eq:logistic_regression} 
with $\ell(y,{\x}^{\top}x) = - y \sigma(\x^\top x) - (1-y) \sigma(\x^\top x)$ (where $\sigma(z) := \frac{1}{1 + e^{-z}}$ denotes the sigmoid function). We use 10 classification datasets from the UCI repository library
and we perform a distributional shift on the train sets: we subsample the majority class so that it accounts for only $10\%$ of the minority class. Then we train a ERM and superquantile 
models. The safety parameter $p$ is tuned via a cross validation procedure on the shifted train set. We finally compute, for the best parameter obtained, the test accuracy and the test loss. 

We report our results in Table~\ref{tab:shift_figures}. For most datasets, we note a significant decrease of the test loss with the superquantile model, when compared to ERM model. In terms of accuracy, the superquantile model offers better performance for this particular distributional shift. 


\begin{table}[!t]
\begin{tabular}{l|cc|cc}
\toprule
 & \multicolumn{2}{c}{Superquantile} & \multicolumn{2}{c}{Expectation} \\ 
 Dataset &  Accuracy  & Loss & Accuracy &  Loss \\ 
\midrule 
Adult& $53.2 \pm 0.67 $ & $0.693 \pm 0.00 $& $\boldsymbol{55.4 \pm 0.48}$& $1.072 \pm 0.01 $\\ 
Monks& $\boldsymbol{64.4 \pm 2.65} $& $0.714 \pm 0.05 $& $54.0 \pm 1.57 $& $1.207 \pm 0.08 $\\ 
Splice& $\boldsymbol{82.7 \pm 0.62} $& $0.681 \pm 0.05 $& $81.7 \pm 0.78 $& $0.557 \pm 0.04 $\\ 
Diabetes& $42.5 \pm 4.72 $& $0.694 \pm 0.00 $& $\boldsymbol{45.1 \pm 4.51}$& $1.325 \pm 0.12 $\\ 
Spambase& $\boldsymbol{78.4 \pm 1.23} $& $0.761 \pm 0.15 $& $77.1 \pm 0.87 $& $0.635 \pm 0.07 $\\ 
Mammography& $\boldsymbol{39.1 \pm 7.59} $& $0.730 \pm 0.01 $& $\boldsymbol{39.1 \pm 6.90} $& $1.293 \pm 0.09 $\\ 
Electricity& $42.8 \pm 0.40 $& $0.693 \pm 0.00 $& $\boldsymbol{47.5 \pm 0.63} $& $1.060 \pm 0.01 $\\ 
Phoneme& $37.3 \pm 5.38 $& $0.737 \pm 0.01 $& $\boldsymbol{50.5 \pm 3.10} $& $1.292 \pm 0.04 $\\ 
Nomao& $\boldsymbol{87.5 \pm 0.22} $& $0.413 \pm 0.03 $& $\boldsymbol{87.4 \pm 0.23} $& $0.394 \pm 0.02 $\\ 
Skin-segmentation& $\boldsymbol{92.1 \pm 0.11} $& $0.420 \pm 0.00 $& $\boldsymbol{91.9 \pm 0.05} $& $0.537 \pm 0.01 $\\ 
\bottomrule
\end{tabular}
\caption{Comparison of performances between a superquantile model and a risk-neutral model for a logistic regression on a distributionally shifted dataset.}
\label{tab:shift_figures}
\end{table}

\paragraph{Robustness to all possible distributional shifts.}

We take the same setting as before, focusing on the \texttt{splice} dataset, and now we perform a sequence of distributional shifts on the training set by rebalancing all the proportions of the two classes. More precisely, for a fixed $\alpha \in (0,1)$, we compute the number $n_{\text{min}}$ of samples from the minority class; we randomly select $\left \lceil \alpha n_{\text{min}} \right \rceil$ points from the majority class and $\left \lceil (1 - \alpha) n_{\text{min}} \right \rceil$ from the minority class. We train on the shifted train set the two logistic regression models of~\eqref{eq:logistic_regression}. We repeat this experiment for 5 different seeds and we compute the average test losses and test accuracies of both models. The experiment is conducted for $100$ values of $\alpha$ evenly spread on $(0,1)$.

The histograms of Figure~\ref{fig:classification} depicts the performances,  as $\alpha$ varies, of ERM against the superquantile (for a fixed probability threshold $p$).  In terms of losses, the superquantile model brings better performances for almost all values of $p$. In particular, the $90^{\textrm{th}}$ quantile of the losses over all considered shifts gets notably decreased for $p$. In terms of accuracy, the superquantile models brings better performance with respect to distributional shifts for all values of $p$. Such behaviours are also observed with other datasets, as those depicted in Figures\;\ref{fig:classification_losses_appendix} and\;\ref{fig:classification_accuracies_appendix} in Appendix.

\begin{figure}[h!]
    \centering
   \includegraphics[width=\linewidth]{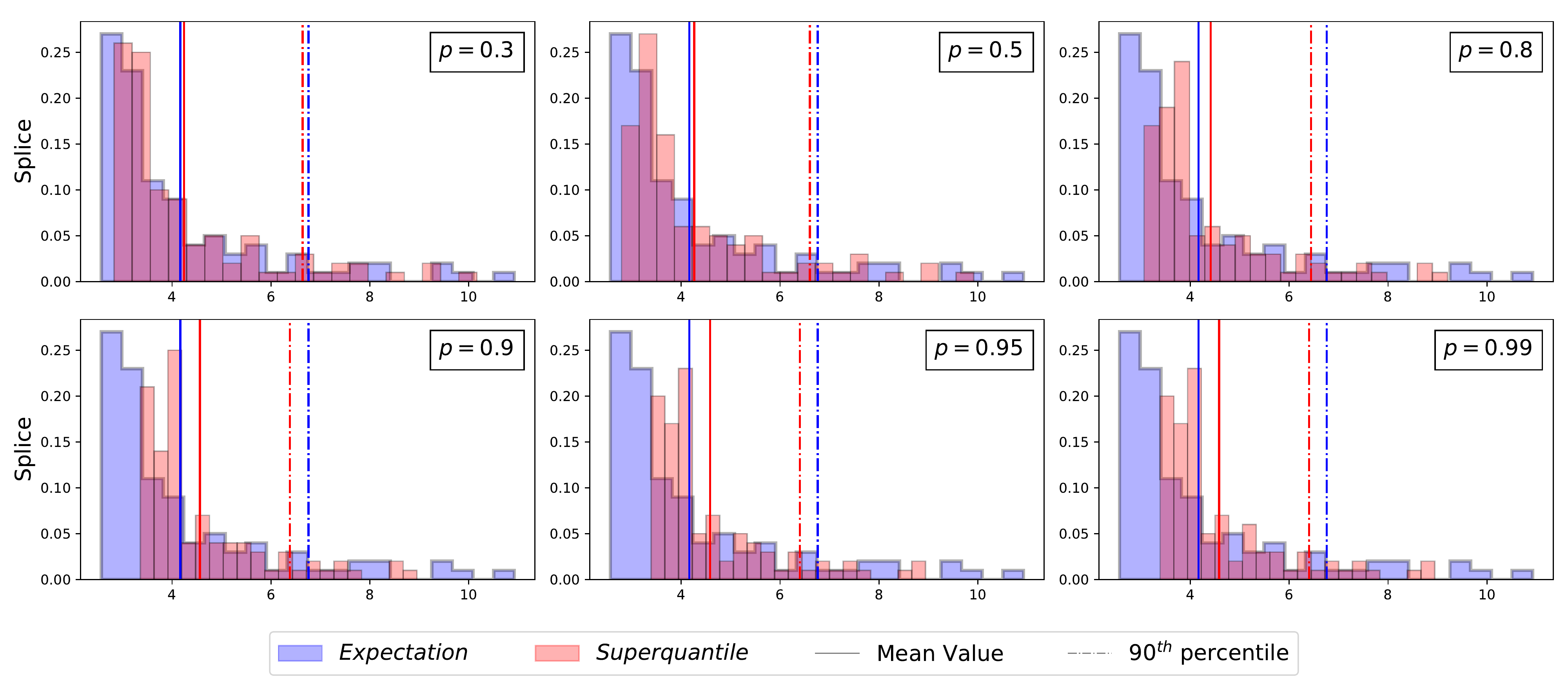}
  
   \includegraphics[width=\linewidth]{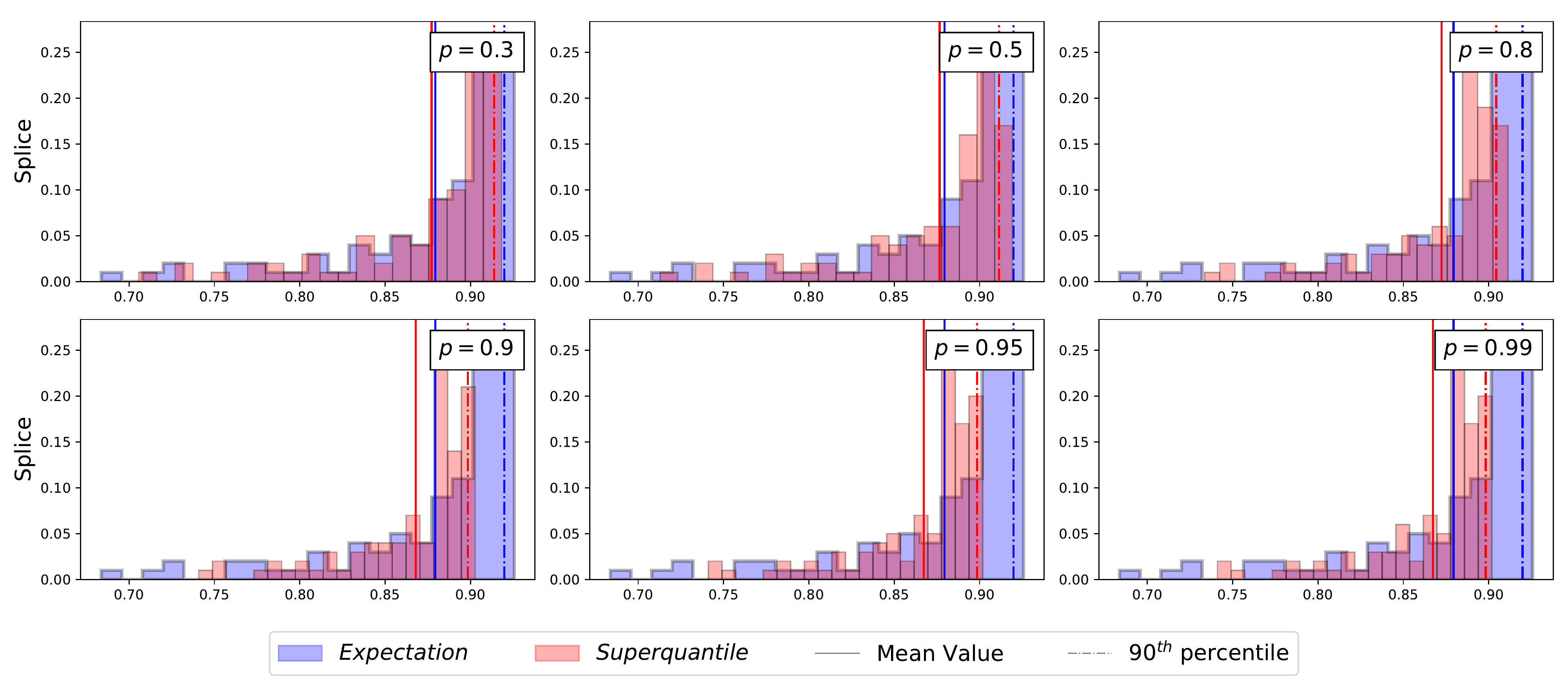}
\caption{Reshaping of histograms of test losses (top) and test accuracies (bottom) over all class imbalances (for a classification task with  logistic regression and the \texttt{splice} dataset). \label{fig:classification}}
\end{figure}

\section{Conclusion, perspectives}
Risk-sensitive optimization can play an important role in the design of safer machine learning models involved in automated decision making. We provide here a software package to tackle superquantile-based learning problems using standard first-order optimization algorithms. The software package is publicly available on the authors' websites.
We have described the main components of the optimization algorithms and how they can be made to tackle superquantile-based learning problems using smoothing techniques in particular. We would tend to recommend the use of a combination of smoothed oracles and batch gradient algorithms to experiment with superquantile-based objectives.

This work can be included in the more general research stream on developing operational tools for distributionally robust learning, which has recently gained interest and focus in the machine learning community; see e.g.\;the recent textbook 
\cite{chen2020distributionally}.
Recent work on related topics developing optimization algorithms with improved complexity bounds~\cite{curi2020adaptive,levy2020large}, exploring fairness challenges~\cite{willaimson2019fairness}, tackling data heterogeneity problems~\cite{laguel:device}, shows the burst of activity in this general area and suggests a number of venues for future investigation.

\begin{acknowledgements}
We acknowledge support from ANR-19-P3IA-0003 (MIAI -- Grenoble Alpes), as well as NSF DMS 2023166, DMS 1839371, CCF 2019844, CIFAR LMB, and faculty research awards. 
\end{acknowledgements}
\clearpage

%
%


\bibliographystyle{plainnat}
\bibliography{references}

\appendix

\section{Additional numerical results}


\begin{figure}[h!]
\centering
\begin{subfigure}[b]{\textwidth}
    \centering
   \includegraphics[width=\linewidth]{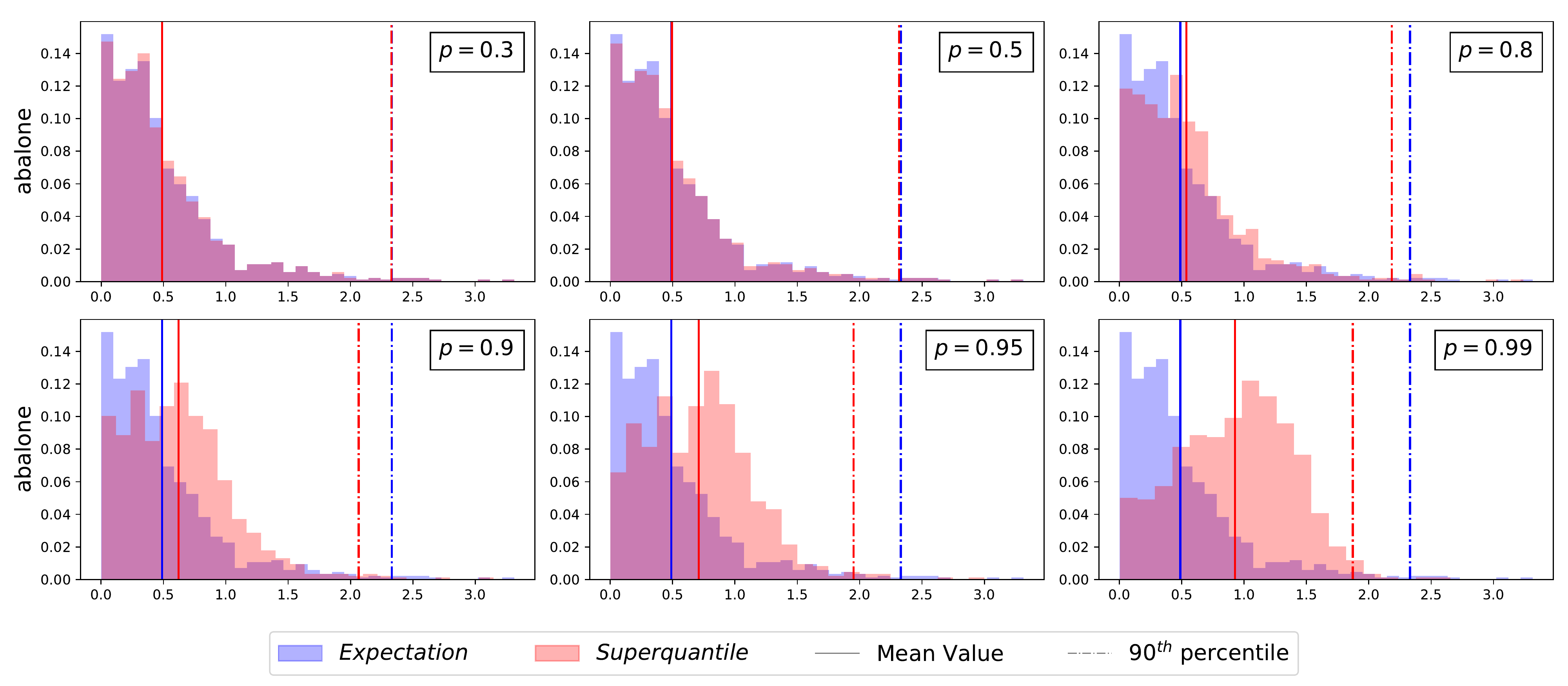}
   \label{fig:abalone} 
\end{subfigure}

\vspace{-20pt}
\begin{subfigure}[b]{\textwidth}
    \centering
   \includegraphics[width=\linewidth]{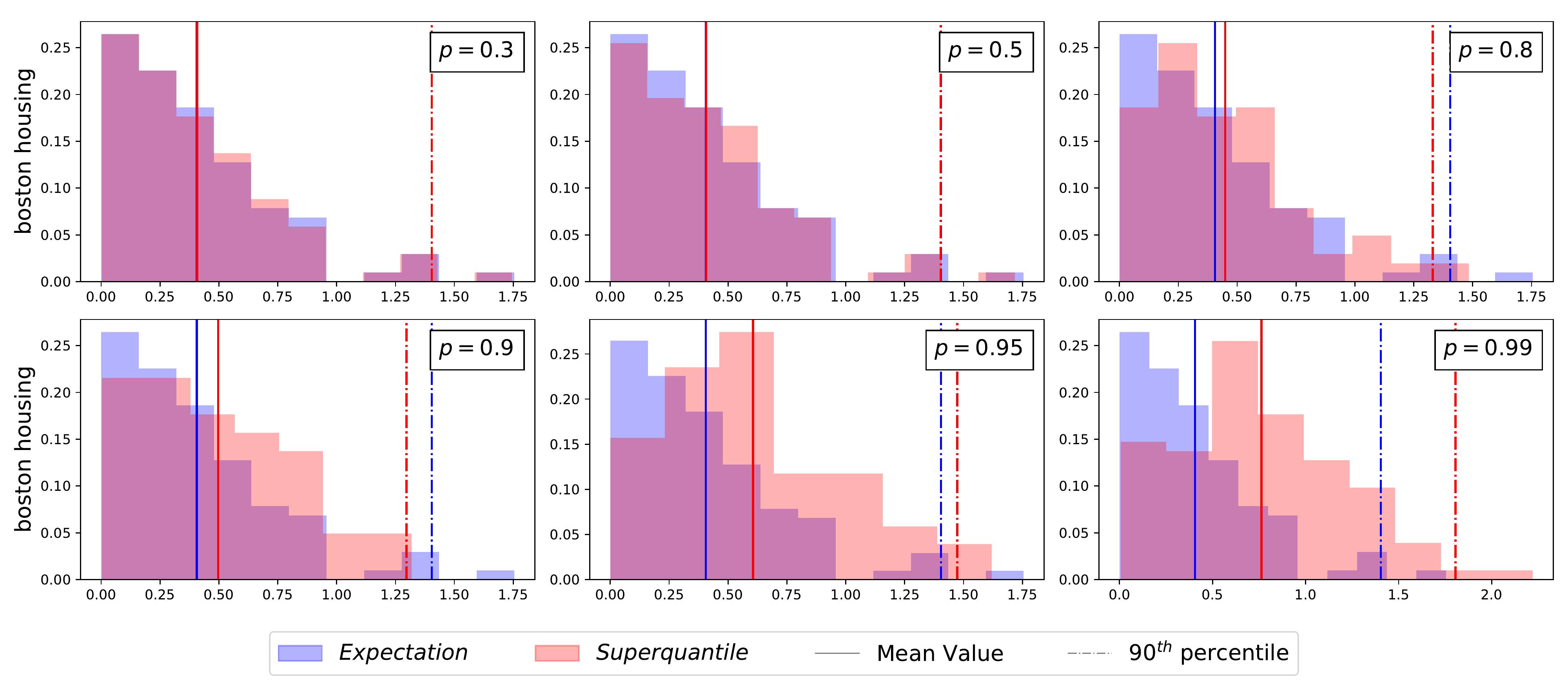}
   \label{fig:boston}
\end{subfigure}

\caption{Ridge regression: comparison of performances between a superquantile model and a ERM model for \texttt{Abalone} and \texttt{Boston Housing}. 
\label{fig:regression_appendix}}
\end{figure}

In this Appendix section, we collect additional results comparing classical supervised learning and superquantile-based supervised learning using the optimization algorithms described in the main text. The experimental setting is exactly the one of Section\;\ref{sec:exp2}, yet we consider here other datasets from UCI repository. The results obtained are essentially the same as the ones presented in Section\;\ref{sec:exp2}, suggesting a greater control of extreme losses and a greater robustness to distributional shift of superquantile-based supervised learning. We refer to the main text for the discussions on these main observations, and we give here additional observations.

Figure \ref{fig:regression_appendix} shows the same behaviour as in Figure\;\ref{fig:regression}:  as the probability threshold $p$ grows, the right tail distribution of losses on the test set gets shifted to the left. We can indeed see, on the subfigures, the reshaping of the histogram and the translation of the $90^{\textrm{th}}$ quantile to the left. Two exceptions should be noticed though:  for the dataset \texttt{boston housing} with $p=0.95$ and $0.99$, the superquantile approach was not able decrease the $90^{\textrm{th}}$ quantile (see the last two subplots at the bottom). This would suggests to avoid in general using too large values of $p$ that would restrict the computational effort to a too small fraction of extreme scenarios only.

Figures\;\ref{fig:classification_losses_appendix} and\;\ref{fig:classification_accuracies_appendix} show results similar to the ones presented in Figure\;\ref{fig:classification} about the resistance to distributional shifts. The three datasets considered here provide a variety of histograms shapes. We see on Figure\;\ref{fig:classification_losses_appendix} that the superquantile brings better performances on the worst-case test losses for all values of $p$ (except for the \texttt{skin-segmentation} with $p \geq 0.9$). Similarly, on Figure\;\ref{fig:classification_accuracies_appendix}, we see, in most cases, improvements of the worst-case test accuracy: sometimes the improvement is important (e.g.\;\texttt{Australian} with $p=0.9$), sometimes it is more marginal or even negative (e.g.\;\texttt{monks-problem1} with $p=0.9$). 

Interestingly, one observes that, for each dataset, there is a particular value of $p$ (depending on the dataset) for which the histogram of losses
 gets shifted to the left uniformly ($p=0.8$ for \texttt{monks-problem-1} and $p=0.3$ for the \texttt{skin-segmentation} dataset). This highlights the importance of a careful tuning of $p$ to address the worst-case outcomes.

\begin{figure}[h!]
\centering
\begin{subfigure}[b]{\textwidth}
    \centering
   \includegraphics[width=\linewidth]{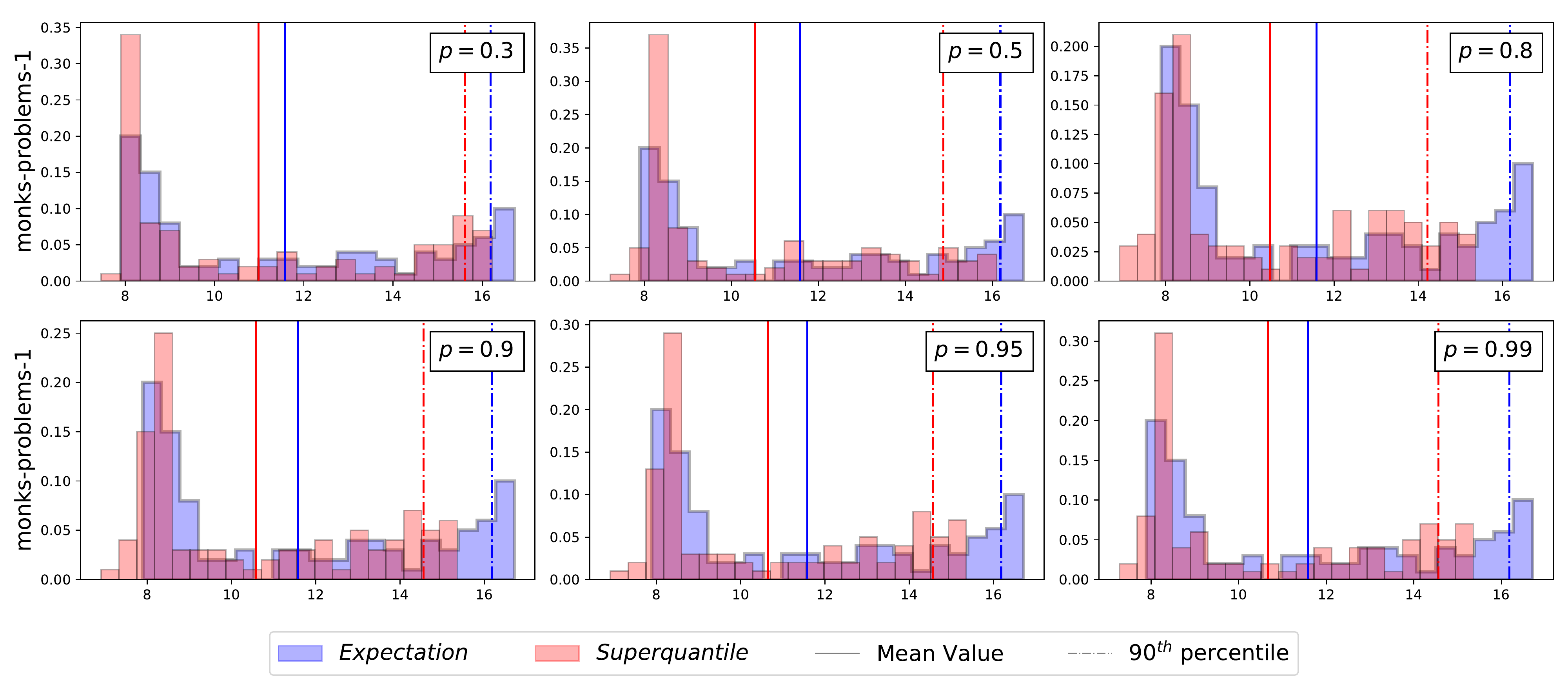}
   \label{fig:loss-monks} 
\end{subfigure}

\vspace{-20pt}
\begin{subfigure}[b]{\textwidth}
    \centering
   \includegraphics[width=\linewidth]{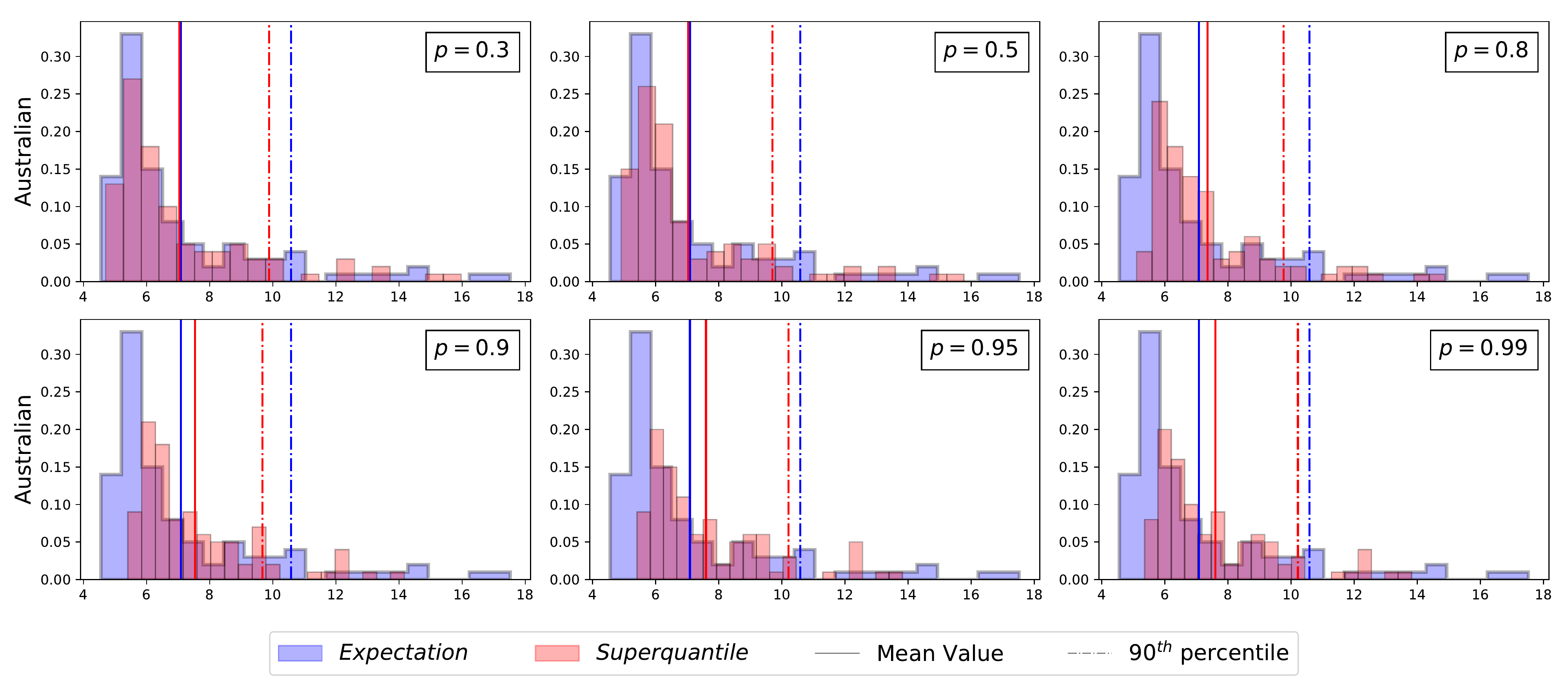}
   \label{fig:loss-aust} 
\end{subfigure}

\vspace{-20pt}
\begin{subfigure}[b]{\textwidth}
    \centering
   \includegraphics[width=\linewidth]{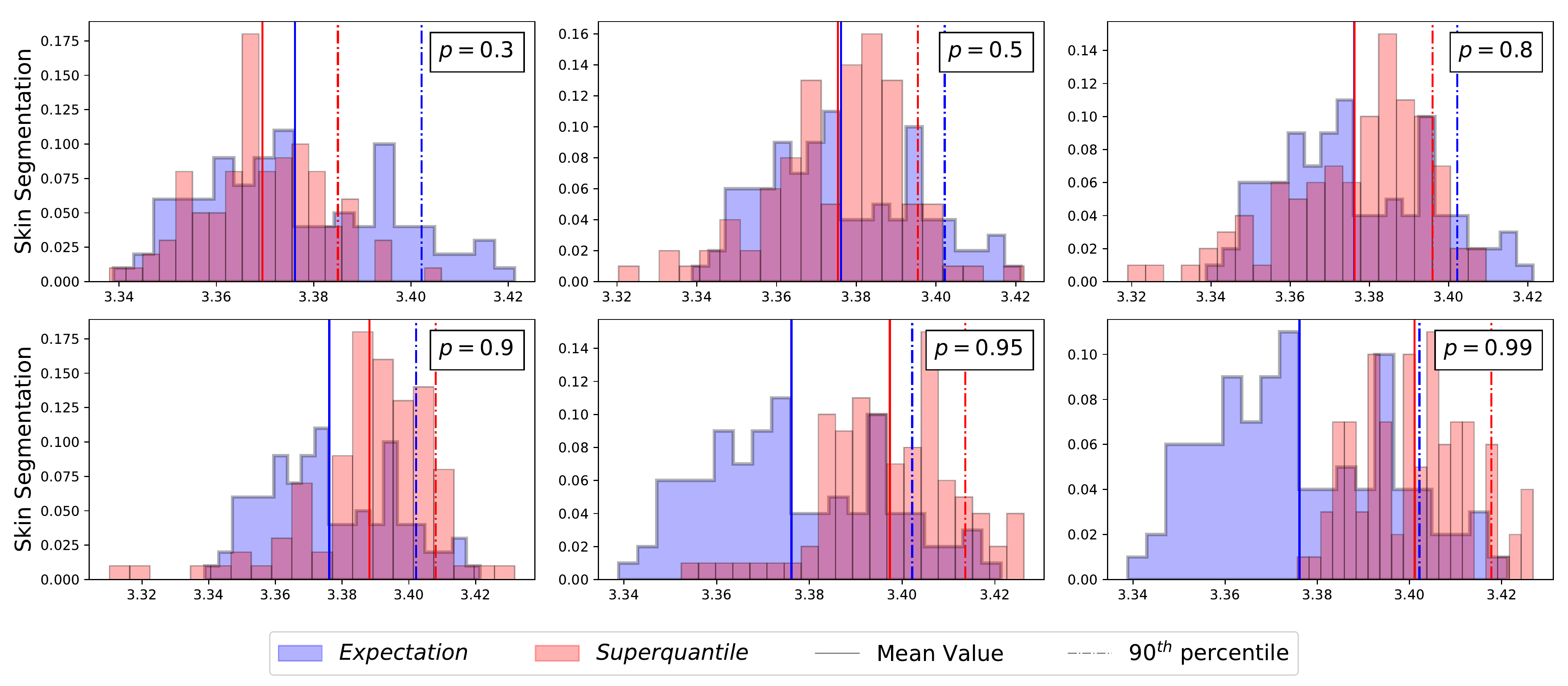}
   \label{fig:loss-skin}
\end{subfigure}
\vspace{-12mm}
\caption{Histogram of test \underline{losses} over all distributional shifts for the datasets \texttt{monks-problem-1}, \texttt{australian-credit}, and \texttt{skin-segmentation}.
}
\label{fig:classification_losses_appendix}
\end{figure}

\begin{figure}[h!]
\centering
\begin{subfigure}[b]{\textwidth}
    \centering
   \includegraphics[width=\linewidth]{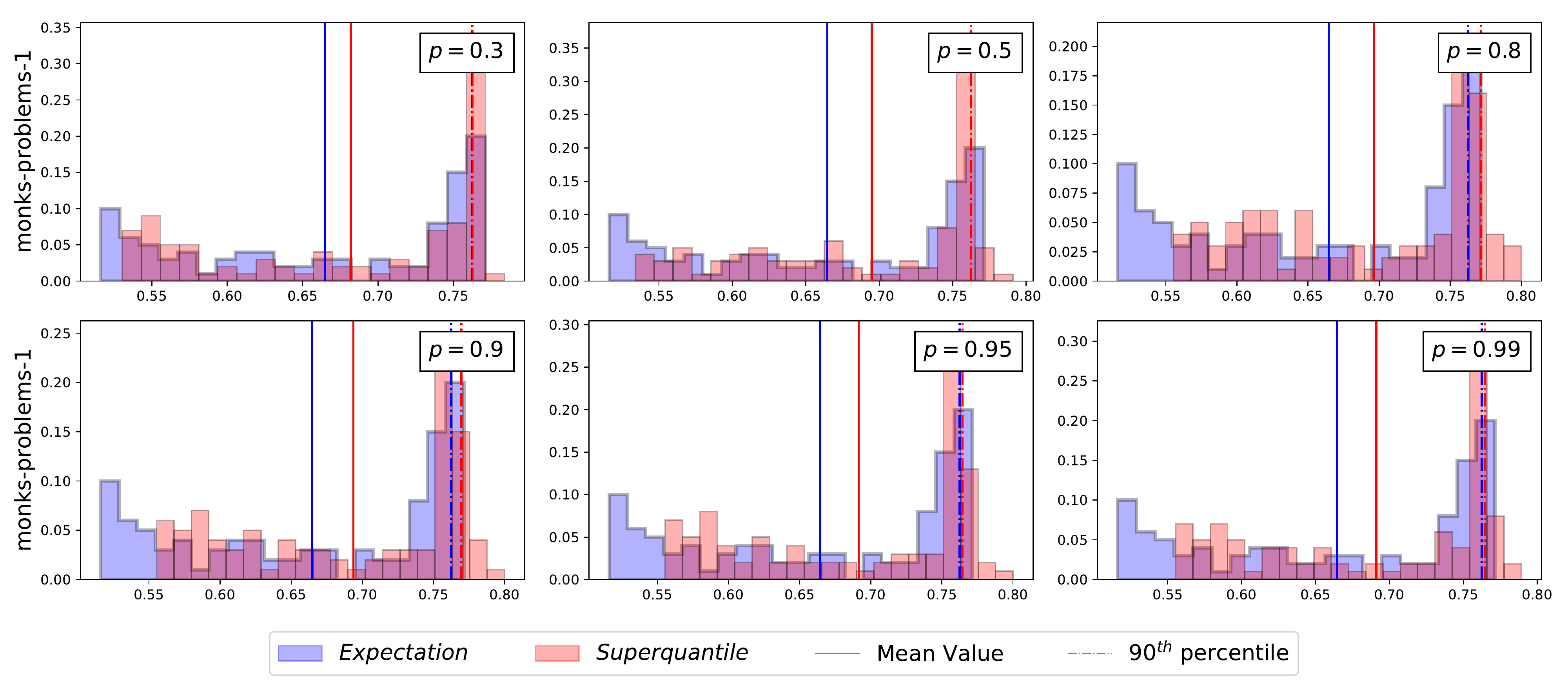}
   \label{fig:acc-monks} 
\end{subfigure}

\vspace{-20pt}
\begin{subfigure}[b]{\textwidth}
    \centering
   \includegraphics[width=\linewidth]{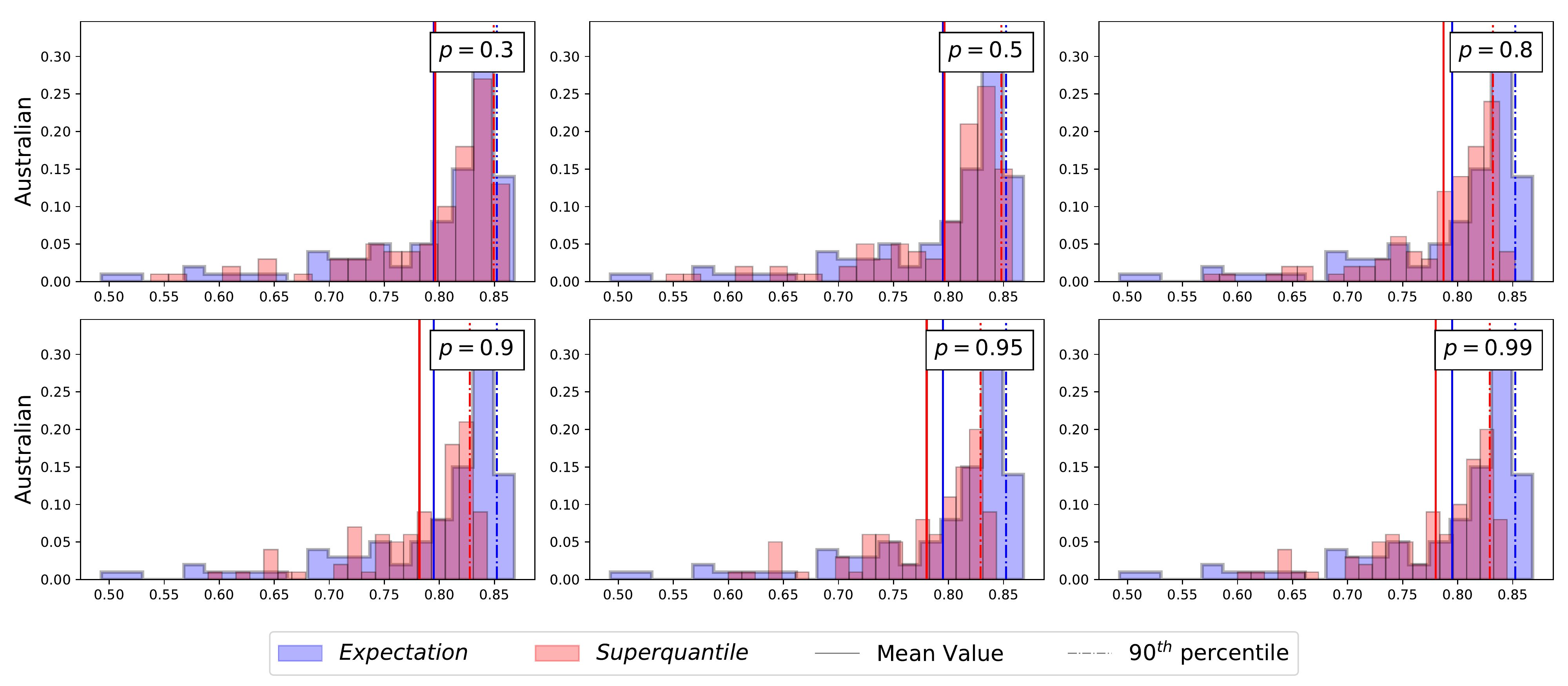}
   \label{fig:acc-aust} 
\end{subfigure}

\vspace{-20pt}
\begin{subfigure}[b]{\textwidth}
    \centering
   \includegraphics[width=\linewidth]{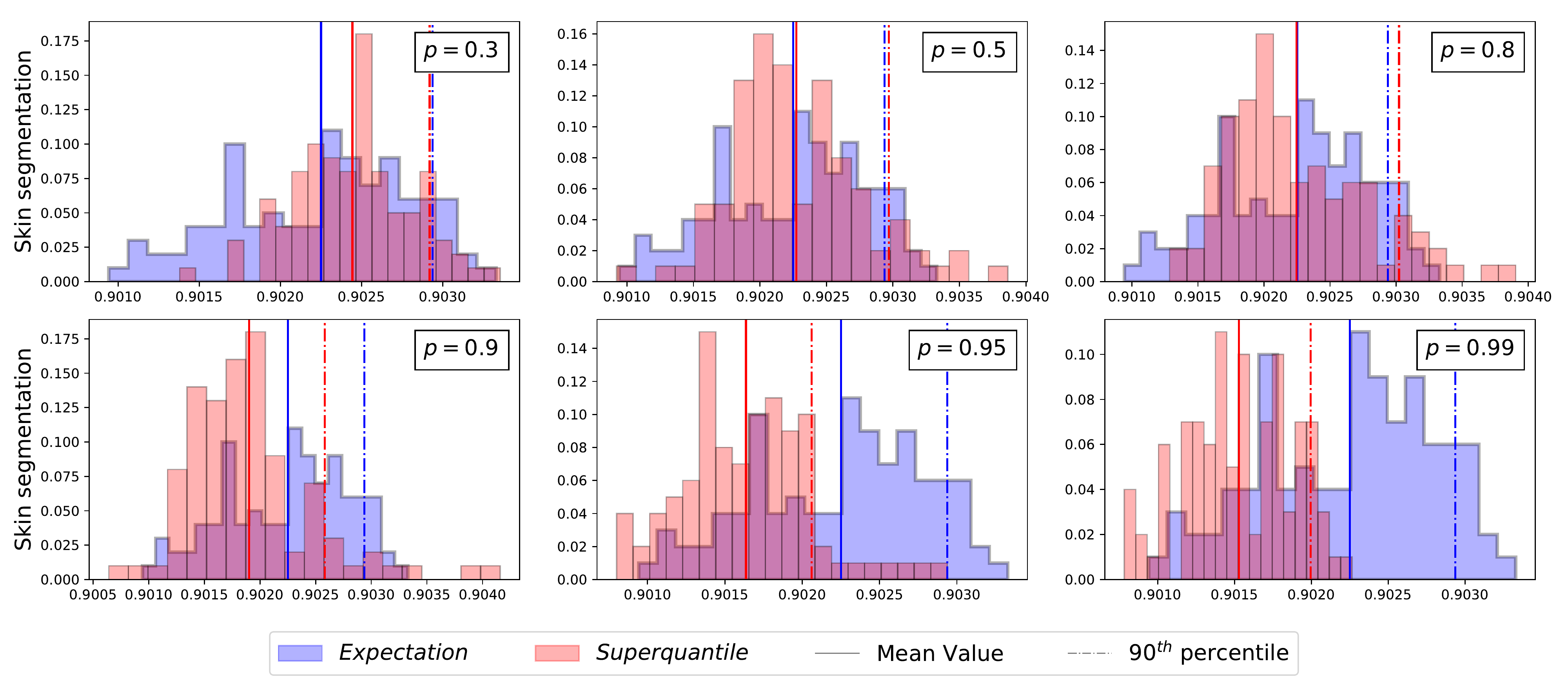}
   \label{fig:acc-skin}
\end{subfigure}
\vspace{-12mm}
\caption{Histogram of test \underline{accuracy} over all distributional shifts for the datasets \texttt{monks-problem-1}, \texttt{australian-credit}, and \texttt{skin-segmentation}.}
\label{fig:classification_accuracies_appendix}
\end{figure}


\end{document}